\documentclass[pdflatex,sn-mathphys-num]{sn-jnl}
\usepackage{graphicx}
\usepackage{booktabs}
\usepackage{algorithm}
\usepackage{algpseudocode}
\usepackage{amsmath,amssymb,amsfonts,mathtools}
\usepackage{tikz}
\usetikzlibrary{arrows.meta,positioning,fit,calc}

\theoremstyle{thmstyleone}
\newtheorem{theorem}{Theorem}[section]
\newtheorem{proposition}[theorem]{Proposition}
\newtheorem{lemma}[theorem]{Lemma}

\newtheorem{assumption}{Assumption}
\theoremstyle{thmstyletwo}
\newtheorem{remark}[theorem]{Remark}
\theoremstyle{thmstylethree}
\newtheorem{definition}[theorem]{Definition}

\newcommand{\dif}{\mathrm{d}}
\newcommand{\E}{\mathbb{E}}
\newcommand{\R}{\mathbb{R}}
\newcommand{\Prob}{\mathbb{P}}
\newcommand{\KL}[2]{D_{\mathrm{KL}}\!\left(#1 \,\middle\|\, #2\right)}
\newcommand{\ssv}{\lambda_{W}}

\makeatletter
\providecommand*{\toclevel@algorithm}{3}
\makeatother

\begin{document}

\title[Entropic SDDP with HJB cross-certification]{InFlow: entropic stochastic
dual dynamic programming with HJB cross-certification for seasonal energy storage}

\author*[1,2]{\fnm{Steeven B.} \sur{Affognon}}
\email{steevenbelvinos@gmail.com}

\author[1]{\fnm{Babacar M.} \sur{Ndiaye}}
\email{babacarm.ndiaye@ucad.edu.sn}

\author[3,4]{\fnm{Cheikh M.F.} \sur{Kebe}}
\email{cheikh.kebe@ucad.edu.sn}

\author[1]{\fnm{Pierre} \sur{Mendy}}
\email{pierre.mendy@ucad.edu.sn}

\affil*[1]{\orgdiv{Laboratory Mathematics of Decision and Numerical Analysis},
\orgname{Cheikh Anta Diop University of Dakar},
\orgaddress{\city{Dakar}, \country{Senegal}}}

\affil[2]{\orgdiv{Department of Mathematics},
\orgname{University of Nairobi},
\orgaddress{\city{Nairobi}, \country{Kenya}}}

\affil[3]{\orgdiv{Laboratoire Eau, Energie, Environnement et proc\'ed\'es Industriels (LE3PI)},
\orgname{Universit\'e Cheikh Anta Diop de Dakar (UCAD)},
\orgaddress{\city{Dakar}, \country{Senegal}}}

\affil[4]{\orgname{Centre de Tests des Syst\`emes Solaires (CT2S)},
\orgaddress{\city{Dakar}, \country{Senegal}}}

\abstract{Seasonal storage optimization requires both scalable stochastic
programming and independent numerical checks when inflow dynamics are strongly
periodic and imperfectly specified. We develop \emph{InFlow}, a
verification-oriented architecture coupling periodic stochastic dual dynamic
programming (SDDP), entropic transition robustness, control-space Bellman
annealing, and an independently discretized Hamilton--Jacobi--Bellman (HJB)
benchmark. The transition operator is interpreted exactly as a KL-penalized
worst-case expectation, while a fixed-radius KL ambiguity set is related to it
through Lagrange duality. We prove a uniform soft-Bellman approximation bound
and validity of Gibbs-tilted SDDP cuts. The benchmark uses a normalized
seasonal Cox--Ingersoll--Ross inflow whose usual uniform Feller positivity
guarantee fails in the driest weeks, requiring explicit treatment of the
degenerate boundary. HJB storage gradients and HiGHS storage-balance duals
produce seasonal water values with correlation $0.995$; an exact-cut-only
replication gives $0.996$. Mesh refinement from $61^2$ to $81^2$ HJB states
changes the annual-mean water value by only $0.6\%$ and the peak by less than
$0.1\%$, although the absolute HJB value level converges more slowly. In
out-of-sample simulation with $3{,}000$ independent trajectories and common
random numbers, $\gamma=5$ incurs a $4.6\%$ nominal mean-cost premium while
reducing CVaR$_{90}$ by $11.2\%$; under a dry-season stress with mean inflow
$40\%$ below nominal, the mean difference is statistically indistinguishable
from zero and CVaR$_{90}$ falls by $11.5\%$. A twelve-point $\gamma$-sweep
confirms monotonicity of the risk-adjusted value, a paired timing experiment
bounds the information advantage of the weekly decision convention, and a
two-reservoir cascade reproduces both marginal water values against a grid
oracle with correlation $0.999$ and the predicted upstream/downstream ratio of
two. The cascade does not support a higher-dimensional acceleration claim,
so entropic guidance is kept separate from exact certification. InFlow is a
transparent methodological benchmark, not a calibrated operational reservoir
model.}



\maketitle

\section{Introduction}
\label{sec:intro}

Seasonal storage problems--hydro reservoirs, gas storage, inventory
carried across demand cycles--share a mathematical core: a controlled
storage state with hard capacity bounds, an exogenous stochastic driver
with pronounced seasonality, and a convex operating cost whose minimization
couples every decision to every future season through a single scalar
shadow price. In hydropower scheduling this price is the \emph{water
value}, the marginal value of stored water; resolved over the annual
cycle we call it the \emph{seasonal storage value} (SSV). Utilities
dispatch against tabulated water values produced by stochastic dual
dynamic programming \citep[SDDP;][]{pereira1991}, while the continuous-time
optimal control literature characterizes the same object as the storage
gradient of a viscosity solution to a Hamilton--Jacobi--Bellman (HJB)
equation \citep{crandall1992,fleming2006}. A recent systematic review of
$277$ water-value studies emphasizes uncertainty treatment, physical fidelity,
and transparent reporting as central methodological trade-offs in hydropower
valuation \citep{pavicevic2026water}. The HJB and SDDP communities compute the
same marginal storage object with different machinery, and this paper exploits
that redundancy deliberately: the low-dimensional PDE solution is used as an
independent numerical oracle for the scalable cutting-plane solver.

Our starting point is a reservoir model whose inflow follows a
Cox--Ingersoll--Ross (CIR) square-root diffusion \citep{cir1985} with
time-varying seasonal mean $\theta(t)$. A constant-parameter CIR process
stays strictly positive when the Feller condition $2\kappa\theta \ge
\sigma_0^2$ holds \citep{feller1951}; with a seasonal $\theta(t)$ the
condition must hold \emph{for every} $t$, and in the calibration considered
here it fails in the driest weeks of the year. This is not a technical
nuisance but the central modeling fact: the weeks in which the diffusion
has a degenerate boundary at the origin--so boundary contact cannot be
excluded by the usual uniform Feller condition--are exactly the weeks of
peak demand and peak storage stress. A solver for this problem should
therefore (i) handle the degenerate boundary correctly rather than assume
it away, and (ii) hedge against the model error that the failed condition
signals. Requirement (i) is addressed by upwind treatment of the $q=0$
boundary in the HJB scheme and by positivity-preserving simulation schemes
\citep{alfonsi2005,lord2010}. Requirement (ii) is where relative entropy
enters.

The contribution of the paper is an architecture, which we call
\textbf{InFlow}, in which one variational object--the relative-entropy
(free-energy) functional--appears twice, in dual roles:
on the control space it produces an \emph{entropically smoothed Bellman
operator} (a soft-min) whose temperature $\eta$ is annealed to zero,
giving a differentiable homotopy from a well-conditioned regularized
problem to the exact one, in the spirit of regularized Markov decision
processes \citep{geist2019} and maximum-entropy control
\citep{todorov2009,kappen2005}; on the probability space, the
Donsker--Varadhan variational formula \citep{dupuis1997} identifies the
\emph{entropic risk measure} exactly with a KL-penalized worst-case
expectation. A fixed-radius KL ambiguity set is connected to the same
functional by Lagrange duality, with the multiplier selected to satisfy the
radius when the constraint binds. Its recursive form is
time-consistent; under the law-invariant and regularity conditions studied by
\citet{kupper2009}, the exponential/entropic class is essentially singled out,
while \citet{ruszczynski2010} provides the dynamic-programming framework for
risk-averse Markov control. Entropic risk has already been incorporated in
SDDP and tested on large-scale hydro-thermal scheduling
\citep{dowson2025convex}; hence the novelty claimed here is \emph{not}
entropic SDDP by itself. It is the combination of transition-space robustness
with control-space annealing, explicit treatment of the seasonal degenerate
boundary, and independent HJB/LP-dual cross-certification. The construction
also connects to risk-averse SDDP
\citep{shapiro2011,philpott2013,shapiro2013risk} and
Hansen--Sargent robustness \citep{hansen2008}. Both entropic roles reduce, in
implementation, to a log-sum-exp. To our knowledge, InFlow is among the first
hydropower-storage frameworks to cross-certify the same seasonal marginal
water value independently as a continuous-time HJB storage gradient and as the
storage-balance dual of a periodic SDDP formulation, while explicitly treating
seasonal diffusion degeneracy and KL-penalized transition uncertainty.

Concretely, the paper makes five contributions. First, we distinguish the
two mathematically different uses of relative entropy and prove the two
results on which the implementation rests: a sandwich/contraction bound for
the soft Bellman operator (Proposition~\ref{prop:soft}) and validity of
Gibbs-tilted cutting planes for the nested entropic continuation value
(Lemma~\ref{lem:cut}). Second, we formulate a steady-seasonal Markov-chain
SDDP scheme \citep{shapiro2020} whose stage problems are solved by HiGHS
\citep{huangfu2018}, with water value read directly from the storage-balance
LP dual. Third, we make the certification discipline explicit: entropic
annealing may guide state exploration, whereas only exact LP-generated cuts
enter a certified lower-bound store. Fourth, we cross-check the resulting
water values against an independent state-constrained HJB discretization,
including mesh/truncation sensitivity and an exact-cut-only replication.
Fifth, we quantify the effect of KL-penalized robustness out of sample,
across a twelve-point robustness sweep, and test the timing convention and a
two-reservoir cascade. The resulting contribution is therefore not a new LP
solver or entropic risk measure in isolation, but a verifiable architecture
linking robust SDDP, water-value duals, and an independent continuous-time
benchmark.

The synthetic design is a deliberate first step in a two-stage research
programme. By matching the continuous and discrete formulations state for
state, the present paper isolates numerical error and certifies the solver
before site-specific complexity is introduced. The empirical hydrological
layer of the MOSSHOOS programme has separately characterized the multiscale
seasonality, annual timing, cross-site coherence, and nonstationarity of
nearly six decades of daily inflows at Bafing Makana, F\'elou, and Gouina
\citep{affognon2026wavelet}. That study provides the data-driven inflow
structure and a reproducible basis for stochastic scenario generation;
operational calibration of InFlow to the Senegal river basin remains a
subsequent stage, beginning with Manantali and later representing the OMVS
cascade, operating rules, travel times, head effects, and physical and
monetary units. The two studies therefore have distinct claims: the companion
hydrological study establishes the empirical seasonal structure, whereas
InFlow provides the methodological and verification layer needed before
basin-specific optimization. No Senegal river operating conclusion is drawn
from the normalized results reported here.

We position InFlow relative to general-purpose optimizers explicitly:
it is not a rival LP kernel to HiGHS or Gurobi but a value-function layer
above one. The deterministic equivalent of the problems treated here grows
exponentially in the scenario tree and is not practically solvable by any
monolithic LP/MIP solver; SDDP-type methods never form it. The comparison
class is therefore SDDP implementations and HJB solvers, and the novelty
is the entropic regularization/robustification of the former certified
against the latter.

The remainder follows the standard scientific sequence requested for the
final manuscript. Section~\ref{sec:methods} gives the model, HJB benchmark,
entropic formulation, periodic SDDP algorithm, certification discipline, and
numerical protocol. Section~\ref{sec:results} reports cross-certification,
robustness, sensitivity, timing, and cascade results. Section~\ref{sec:discussion}
interprets the findings and states the limitations and operational scope.
Section~\ref{sec:conclusion} concludes. Proofs and discretization details are
collected in the appendices.

\begin{figure}[t]
\centering
\resizebox{\linewidth}{!}{%
\begin{tikzpicture}[
  font=\sffamily\scriptsize,
  node distance=5mm and 7mm,
  flow/.style={-{Latex[length=2mm]}, line width=.55pt, draw=black!65},
  data/.style={rounded corners=1.5mm, draw=blue!55!black, fill=blue!7,
    line width=.55pt, align=center, inner sep=4pt, text width=34mm},
  entropy/.style={rounded corners=1.5mm, draw=teal!60!black, fill=teal!7,
    line width=.55pt, align=center, inner sep=4pt, text width=34mm},
  solver/.style={rounded corners=1.5mm, draw=orange!70!black, fill=orange!8,
    line width=.65pt, align=center, inner sep=4pt, text width=36mm},
  check/.style={rounded corners=1.5mm, draw=violet!60!black, fill=violet!7,
    line width=.55pt, align=center, inner sep=4pt, text width=34mm},
  outcome/.style={rounded corners=1.5mm, draw=black!60, fill=black!4,
    line width=.55pt, align=center, inner sep=4pt, text width=35mm},
  group/.style={rounded corners=2mm, draw=black!25, dashed, inner sep=4mm}
]

\node[data] (physical) {\textbf{Seasonal storage system}\\
CIR inflow $Q_t$, storage $S_t$, demand $D_t$};
\node[data, right=of physical] (chain) {\textbf{Discrete stochastic model}\\
weekly conditional-mean Markov chain};

\node[entropy, below left=8mm and -5mm of physical] (control)
{\textbf{Control-space entropy}\\
soft Bellman operator; $\eta\downarrow0$};
\node[entropy, below right=8mm and -5mm of chain] (transition)
{\textbf{Transition-space entropy}\\
nested entropic risk; parameter $\gamma$};

\node[solver, below=10mm of $(control)!0.5!(transition)$] (phasea)
{\textbf{Phase A: smooth global initialization}\\
annealed value iteration and conservative seed cuts};
\node[solver, below=of phasea] (phaseb)
{\textbf{Phase B: exact periodic SDDP}\\
HiGHS stage LPs and entropically tilted cuts};
\node[outcome, below=of phaseb] (dual)
{\textbf{Scalable storage value}\\
$\lambda_W^{\mathrm{SDDP}}=-\mu$ from the LP balance dual};

\node[check, left=12mm of phaseb] (hjb)
{\textbf{Independent HJB benchmark}\\
degenerate-boundary finite differences};
\node[outcome, below=of hjb] (grad)
{\textbf{Continuous storage value}\\
$\lambda_W^{\mathrm{HJB}}=-\partial_sV$};

\node[check, below=9mm of $(grad)!0.5!(dual)$] (cert)
{\textbf{Cross-certification}\\
gradient--dual agreement and bias detection};
\node[outcome, right=of cert] (oos)
{\textbf{Out-of-sample assessment}\\
nominal/stressed mean cost and CVaR$_{90}$};

\draw[flow] (physical) -- (chain);
\draw[flow] (physical.south) |- (control.north);
\draw[flow] (chain.south) |- (transition.north);
\draw[flow] (control) -- (phasea);
\draw[flow] (transition) -- (phasea);
\draw[flow] (phasea) -- (phaseb);
\draw[flow] (phaseb) -- (dual);
\draw[flow] (physical.west) -- ++(-8mm,0) |- (hjb.west);
\draw[flow] (hjb) -- (grad);
\draw[flow] (grad) -- (cert);
\draw[flow] (dual) -- (cert);
\draw[flow] (cert) -- (oos);
\draw[flow] (phaseb.east) -| (oos.north);

\node[group, fit=(control)(transition)(phasea)] {};

\end{tikzpicture}%
}
\caption{InFlow architecture. The seasonal physical model is discretized
into a conditional-mean inflow chain. Relative entropy regularizes the
control minimization and robustifies transition aggregation before exact
periodic SDDP refinement. The resulting LP-dual water value is checked
against an independently discretized HJB storage gradient, after which
policies are evaluated under nominal and misspecified inflow dynamics.}
\label{fig:architecture}
\end{figure}
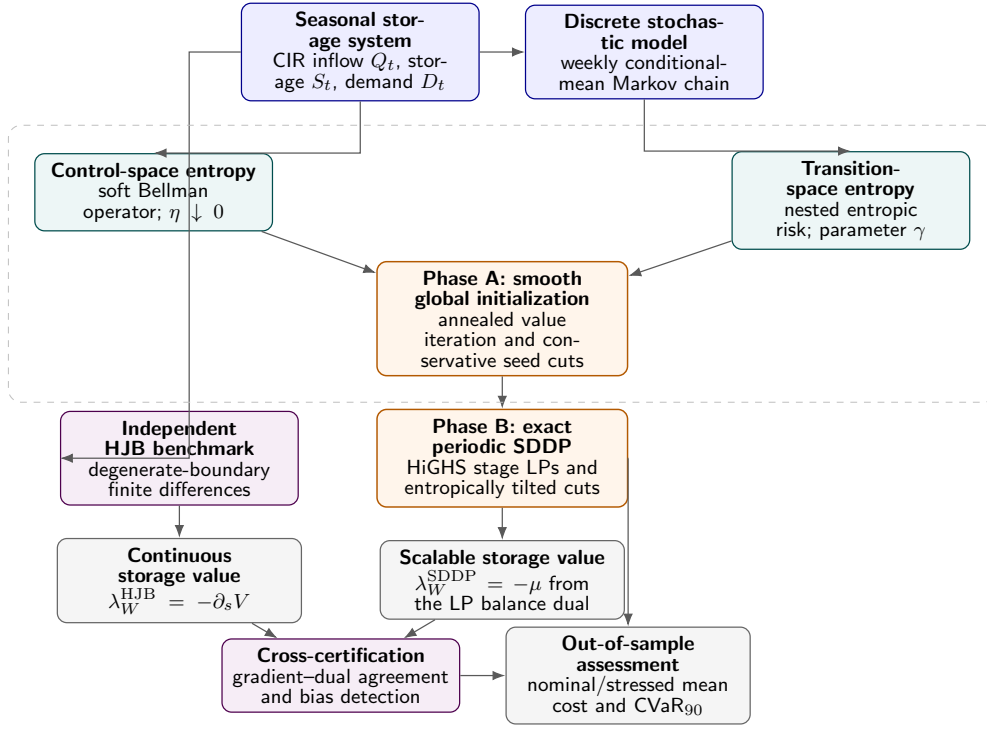

\section{Methods}
\label{sec:methods}

\subsection{Seasonal storage model and water value}
\label{sec:model}
\label{sec:hjb}

Time $t$ is measured in years. The state is $(S_t,Q_t)$, storage and
inflow. Controls are the release rate $u_t \in [0,U_{\max}]$ and the spill
rate $w_t \ge 0$. The dynamics are
\begin{align}
\dif S_t &= (Q_t - u_t - w_t)\,\dif t, \qquad 0 \le S_t \le S_{\max},
\label{eq:stor}\\
\dif Q_t &= \kappa\left[\theta(t) - Q_t\right]\dif t
 + \sigma_0 \sqrt{Q_t}\,\dif W_t, \qquad Q_t \ge 0,
\label{eq:cir}
\end{align}
with $\theta(t) = \bar\theta\left(1 + a\cos 2\pi(t - t_w)\right)$ the
seasonal inflow mean and $W$ a standard Brownian motion. Demand
$D(t) = \bar D\left(1 + b\cos 2\pi(t - t_d)\right)$ peaks
counter-seasonally to inflow ($t_d \ne t_w$). The running cost is the
thermal cost of the residual load plus a small spill penalty,
\begin{equation}
\ell(t,u,w) \;=\; c_1\,x + \tfrac{c_2}{2}\,x^2 \;+\; \varepsilon_w w,
\qquad x = \left(D(t)-u\right)^+,
\label{eq:cost}
\end{equation}
and the objective is the discounted infinite-horizon (equivalently,
periodic) cost
\begin{equation}
V(t,s,q) \;=\; \inf_{(u,w)\in\mathcal U(s,q)}
\E_{t,s,q}\!\left[\int_t^{\infty} e^{-\rho(\tau-t)}\,
\ell(\tau,u_\tau,w_\tau)\,\dif\tau\right],
\label{eq:objective}
\end{equation}
over admissible controls adapted to the filtration of $W$ and respecting
the state constraints in \eqref{eq:stor}; spill is the control induced at
the upper storage boundary. Because all data are $1$-periodic in $t$ and
$\rho>0$, $V$ is the unique $1$-periodic fixed point of the one-year
dynamic-programming map, and we work with this \emph{steady-seasonal}
value function throughout; it replaces any ad hoc terminal condition
$g(S_T)$ and makes the SSV a property of the cycle rather than of a
truncation point.

Throughout, the \emph{water value} is
\begin{equation}
\ssv(t,s,q) \;=\; -\,\partial_s V(t,s,q) \;\ge\; 0,
\end{equation}
and the \emph{seasonal storage value} is $\ssv$ viewed as a function of
the phase $t$ of the annual cycle.

\begin{assumption}[standing]\label{ass:standing}
$\rho>0$; $\theta,D$ are continuous and $1$-periodic; $c_1,c_2>0$,
$\varepsilon_w\ge 0$; $U_{\max} > \sup_t D(t)$; $S_{\max}>0$.
\end{assumption}

Under Assumption~\ref{ass:standing} the stage costs are nonnegative and
jointly convex in $(s,u,w)$, the dynamics \eqref{eq:stor} are linear in
$(s,u,w)$, and $s \mapsto V(t,s,q)$ is convex and nonincreasing for each
$(t,q)$; these are the structural facts the cutting-plane method uses.

\subsection{HJB equation and release threshold}

Dynamic programming applied to
\eqref{eq:stor}--\eqref{eq:objective} yields, on
$(0,1)\times(0,S_{\max})\times(0,\infty)$ with periodicity in $t$,
\begin{equation}
\rho V - \partial_t V \;=\;
\min_{(u,w)\in\mathcal U(s,q)}
\Big\{ \ell(t,u,w) + (q-u-w)\,\partial_s V \Big\}
+ \kappa\left[\theta(t)-q\right]\partial_q V
+ \tfrac12\sigma_0^2\, q\, \partial_{qq} V .
\label{eq:hjb}
\end{equation}
The value function is in general nonsmooth at the storage bounds, and
\eqref{eq:hjb} is understood in the constrained viscosity sense
\citep{crandall1992,fleming2006}; the state constraints enter as boundary
inequalities (release is limited by availability at $s=0$; spill is
induced at $s=S_{\max}$).

The inner minimization is one-dimensional and separable. Writing
$m(t,u) = -\partial_u \ell = c_1 + c_2\left(D(t)-u\right)$ for $u<D(t)$
(and $0$ beyond) for the marginal avoided system cost of release, the
Karush--Kuhn--Tucker condition for \eqref{eq:hjb} gives the threshold
policy
\begin{equation}
u^*(t,s,q) \;=\;
\begin{cases}
0, & m(t,0) \le \ssv(t,s,q), \\[2pt]
D(t) - \dfrac{\ssv(t,s,q) - c_1}{c_2}, & c_1 < \ssv(t,s,q) < m(t,0),\\[6pt]
D(t), & \ssv(t,s,q) \le c_1,
\end{cases}
\label{eq:threshold}
\end{equation}
clipped to $[0,U_{\max}]$: water is released exactly while its marginal
avoided cost exceeds the water value. Equation~\eqref{eq:threshold} is
the continuous-time twin of the discrete rule ``dispatch until marginal
cost falls to the active SDDP cut slope''; the identity of these two
statements is what Section~\ref{sec:results} verifies numerically.

\subsection{Failure of the Feller condition and its consequences}
\label{sec:feller}

For constant $\theta$, the CIR process \eqref{eq:cir} satisfies
$\Prob(Q_t>0 \;\forall t)=1$ if and only if the Feller ratio
$F = 2\kappa\theta/\sigma_0^2 \ge 1$ \citep{feller1951,cir1985}. With
time-varying $\theta(t)$ no global claim follows unless
$2\kappa\theta(t)\ge\sigma_0^2$ for \emph{every} $t$. In our calibration
($\kappa=8$, $\sigma_0=2$, $\bar\theta=1$, $a=0.8$) the ratio is
$F(t)=4\theta(t)\in[0.8,7.2]$: it falls below one for the five driest
weeks of the year. Three consequences are drawn, and all three are
implemented rather than assumed:

first, the usual uniform condition that excludes boundary contact is
unavailable in the dry season, so the HJB scheme must handle rather than
silently exclude the origin;
the diffusion coefficient $\tfrac12\sigma_0^2 q$ vanishes at $q=0$ while
the drift $\kappa\theta(t)>0$ points inward, and a one-sided upwind
discretization of the drift term is the correct degenerate-boundary
treatment. Second, path simulation must preserve nonnegativity without
relying on Feller: we use the drift-implicit square-root scheme of
\citet{alfonsi2005} whenever its well-posedness guard
$q_k + (\kappa\theta - \tfrac12\sigma_0^2)h \ge 0$ holds, and the
full-truncation Euler scheme of \citet{lord2010}, which converges without
any Feller condition, for the (dry-season, near-zero) steps where it
fails. Third--and this motivates Section~\ref{sec:entropy}--the weeks
in which the transition density degenerates (the noncentral-$\chi^2$
degrees-of-freedom parameter $d = 4\kappa\theta/\sigma_0^2 = 2F$ drops
below $2$, so the density is unbounded at the origin) are the weeks in
which any finite calibration record pins the law of $Q$ least reliably.
A solver that trusts the nominal law exactly there is fragile by
construction. We therefore hedge over a KL neighborhood of the nominal
transition law, sized by a single parameter, rather than over a
parametric family chosen ex ante.

\subsection{Relative entropy and robust dynamic programming}
\label{sec:entropy}

This section is the mathematical core of the paper. Both uses of relative
entropy rest on the same variational identity. For probability measures
$Q \ll P$ on a common space, the relative entropy is
$\KL{Q}{P} = \int \log\frac{\dif Q}{\dif P}\,\dif Q \ge 0$, and for any
bounded measurable $X$ the Donsker--Varadhan (Gibbs) variational formula
holds \citep{dupuis1997}:
\begin{equation}
\frac1\gamma \log \E_P\!\left[e^{\gamma X}\right]
\;=\;
\sup_{Q \ll P}\;\Big\{ \E_Q[X] \;-\; \tfrac1\gamma \KL{Q}{P} \Big\},
\qquad \gamma > 0,
\label{eq:dv}
\end{equation}
with the supremum attained by the Gibbs tilt
$\dif Q^\star/\dif P \propto e^{\gamma X}$. Read left to right,
\eqref{eq:dv} says the log-sum-exp is a KL-penalized worst case; read
right to left, it says every KL-regularized optimization has a
closed-form value. The solver uses \eqref{eq:dv} in both directions: over
the \emph{control} measure (Section~\ref{sec:soft}) and over the
\emph{transition} measure (Section~\ref{sec:robust}).

\subsubsection{Entropic smoothing of the Bellman operator}
\label{sec:soft}

Discretize time into stages of length $\Delta = 1/52$ with per-stage
discount $\delta = e^{-\rho\Delta}$, and let the exogenous driver take
values in a finite set indexed by $j$ (Section~\ref{sec:algo} constructs
this chain). The exact Bellman operator at stage $t$, node $j$, storage
$s$ is
\begin{equation}
(\mathcal T_{t,j} V)(s) \;=\;
\min_{u \in \mathcal U_{t,j}(s)}
\Big\{ c_{t,j}(s,u) + \delta\, \mathcal R_{t,j}\big(V_{t+1,\cdot}(s')\big) \Big\},
\qquad s' = s + (q_{t,j} - u - w)\Delta,
\label{eq:bellman}
\end{equation}
where $\mathcal R$ is the (for now, risk-neutral) aggregation over the next
node, $\mathcal R_{t,j}(v) = \sum_{j'} P_t[j,j']\,v_{j'}$. Fix a reference
probability measure $\pi_0$ on the control set (uniform on a grid of
$N_u$ admissible releases in our implementation). The
\emph{entropic Bellman operator} at temperature $\eta>0$ relaxes the
minimum over controls to a minimum over randomized controls with a KL
penalty against $\pi_0$:
\begin{align}
(\mathcal T^{\eta}_{t,j} V)(s)
&\;=\; \min_{\pi \in \mathcal P(\mathcal U)}
\Big\{ \E_{u\sim\pi}\big[G(s,u)\big] + \eta \KL{\pi}{\pi_0} \Big\}
\label{eq:softdef}\\
&\;=\; -\,\eta \log \E_{u \sim \pi_0}\!\left[
\exp\!\big(-G(s,u)/\eta\big)\right],
\label{eq:softlse}\\[-1mm]
&\text{where}\qquad
G(s,u) := c_{t,j}(s,u)
+ \delta\,\mathcal R_{t,j}\!\big(V_{t+1,\cdot}(s')\big).\notag
\end{align}
where \eqref{eq:softlse} follows from \eqref{eq:dv} applied with
$\gamma = 1/\eta$ and $X = -G$, and the minimizing policy is the Gibbs
measure $\pi^\star(\dif u) \propto \pi_0(\dif u)\,e^{-G(s,u)/\eta}$. This
is the standard regularized-MDP construction \citep{geist2019}, put here
to a specific numerical purpose: \eqref{eq:softlse} is $C^\infty$ in $s$
wherever $G$ is and collapses to
\eqref{eq:bellman} as $\eta \downarrow 0$. The properties we invoke are
elementary but worth recording precisely.

\begin{proposition}[sandwich, monotonicity, contraction]
\label{prop:soft}
Let $\pi_0$ be uniform on a finite control grid of $N_u$ points, and let
$V, V'$ be bounded. Then, pointwise,
\begin{enumerate}
\item[(i)] $\displaystyle \mathcal T V \;\le\; \mathcal T^{\eta} V
\;\le\; \mathcal T V + \eta \log N_u$;
\item[(ii)] $V \le V' \implies \mathcal T^{\eta} V \le \mathcal T^{\eta} V'$,
and $\mathcal T^\eta$ is a $\delta$-contraction in the sup norm;
\item[(iii)] the fixed points satisfy
$\displaystyle 0 \;\le\; V^{\eta} - V^{*} \;\le\; \frac{\eta \log N_u}{1-\delta}$,
where $V^*$ and $V^\eta$ are the (periodic) fixed points of $\mathcal T$
and $\mathcal T^\eta$.
\end{enumerate}
\end{proposition}

\begin{proof}
(i) With $\pi_0$ uniform,
$\E_{\pi_0} e^{-G/\eta} \le e^{-\min G/\eta}$ gives the left inequality
after taking $-\eta\log$; and
$\E_{\pi_0} e^{-G/\eta} \ge \tfrac1{N_u} e^{-\min G/\eta}$ gives the
right. (ii) Monotonicity is inherited from monotonicity of
$G \mapsto -\eta\log\E_{\pi_0}e^{-G/\eta}$; the contraction factor is
$\delta$ because a constant shift $c$ in $V$ shifts $G$ by $\delta c$ and
the log-sum-exp is translation-equivariant. (iii) Apply (i) and (ii) to
the two fixed-point equations and sum the geometric series.
\end{proof}

Proposition~\ref{prop:soft}(iii) is the annealing license: solving at
temperature $\eta$ costs at most $\eta\log N_u/(1-\delta)$ in value, so a
schedule $\eta_1 > \eta_2 > \cdots \downarrow 0$ interpolates from a
smooth, well-conditioned problem--in which gradients (hence water
values) are stable and the state space is explored under a dispersed
Gibbs policy--to the exact problem, with a computable bound at every
stage of the homotopy. In the numerical study the annealed water value
visibly sharpens onto the exact one (Figure~\ref{fig:annealing}).

\begin{remark}[on exact linearization]\label{rem:hopfcole}
When the control cost \emph{is} the relative entropy between controlled
and passive path measures and the control acts through the noise
channels, the exponential transform $\psi = e^{-V/\eta}$ linearizes the
HJB equation and $V$ admits a Feynman--Kac/path-integral representation
\citep{todorov2009,kappen2005}. We record explicitly that the storage
problem \emph{fails both requirements}: the dispatch cost
\eqref{eq:cost} is an economic datum that cannot be replaced by a KL
penalty, and the control acts on the noiseless storage equation
\eqref{eq:stor} while the noise lives in the uncontrolled inflow equation
\eqref{eq:cir}. Exact linearization is therefore structurally unavailable
here; what survives--by design--is the smoothing role of
Proposition~\ref{prop:soft}. Claims of Hopf--Cole linearizability for
problems of this class should be treated with the same caution as global
Feller claims.
\end{remark}

\subsubsection{KL-penalized transition robustness and entropic risk}
\label{sec:robust}

The second use of \eqref{eq:dv} replaces the risk-neutral aggregation
$\mathcal R_{t,j}(v)=\sum_{j'}P_t[j,j']v_{j'}$ by the multiplier-robust
functional
\[
\rho_{\gamma}^{t,j}(v)
=\sup_{Q\ll P_t[j,\cdot]}
\left\{\E_Q[v]-\frac{1}{\gamma}\KL{Q}{P_t[j,\cdot]}\right\}.
\]
Thus a fixed $\gamma$ is exactly a \emph{KL-penalty multiplier}, not a
fixed ambiguity radius. If instead a radius $r>0$ is prescribed, define
$\mathcal R^{\mathrm{rob}}_{t,j}(v)=\sup\{\E_Q[v]:\KL{Q}{P_t[j,\cdot]}\le r\}$.
Lagrangian duality gives, for every $\gamma>0$,
\begin{equation}
\mathcal R^{\mathrm{rob}}_{t,j}(v)
\;\le\;
\rho_{\gamma}^{t,j}(v) + \frac{r}{\gamma},
\qquad
\rho_{\gamma}^{t,j}(v)
\;:=\;
\frac1\gamma \log \sum_{j'} P_t[j,j']\; e^{\gamma v_{j'}},
\label{eq:entropicrisk}
\end{equation}
and minimizing the right-hand side over $\gamma>0$ recovers the
fixed-radius problem under the usual constraint qualification. For the
multiplier formulation used by InFlow, the maximizing law at fixed
$\gamma$ is the Gibbs tilt
\begin{equation}
\tilde w_{j'} \;=\;
\frac{P_t[j,j']\, e^{\gamma v_{j'}}}{\sum_{k} P_t[j,k]\, e^{\gamma v_{k}}},
\label{eq:tilt}
\end{equation}
which overweights precisely the transitions into expensive states. The
functional $\rho_\gamma$ is the \emph{entropic risk measure}: convex,
monotone and translation-equivariant. Nested recursively, it is
time-consistent; moreover, under the law-invariance and regularity
conditions characterized by \citet{kupper2009}, the exponential/entropic
form is essentially singled out among dynamic certainty equivalents
\citep{ruszczynski2010,follmer2016}. Replacing $\mathcal R$ by
$\rho_\gamma$ in \eqref{eq:bellman} therefore yields a \emph{bona fide}
dynamic program, a member of the risk-averse SDDP family
\citep{shapiro2011,philpott2013,dowson2025convex}, and simultaneously the
Hansen--Sargent multiplier formulation of model misspecification robustness
\citep{hansen2008}: $\gamma$ prices model
distrust, and $\gamma \downarrow 0$ recovers the risk-neutral recursion
continuously since $\rho_\gamma(v) = \sum_{j'} P v_{j'} +
\tfrac{\gamma}{2}\operatorname{Var}_P(v) + O(\gamma^2)$.

Two structural observations close the loop with
Section~\ref{sec:feller}. The tilt \eqref{eq:tilt} is state- and
season-dependent: it is largest where the spread of $v_{j'}$ across
successor nodes is largest, which in this problem is the approach to and
interior of the dry season--the same weeks where the degenerating
transition density makes the nominal $P_t$ least reliable. The hedge thus
concentrates automatically where the Feller analysis says the model is
fragile, without specifying a parametric family of alternative transition laws. And computationally,
$\rho_\gamma$ is \emph{again} a log-sum-exp: the robust solver reuses the
identical primitive as the entropic smoothing of
Section~\ref{sec:soft}--once over controls with $\gamma = 1/\eta$, once
over transitions with the risk parameter $\gamma$. One functional, two
dualities, one implementation.

\subsection{Periodic Markov-chain SDDP and water-value duals}
\label{sec:algo}

The exogenous state is discretized into a chain over the \emph{weekly
mean} inflow: bins $B_1,\dots,B_N$ with week-of-year transition matrices
$P_t$ and node values $q_{t,j} = \E[\bar Q_{\text{week}} \mid \text{week } t,\,
\bar Q \in B_j]$, all estimated from long simulations of
\eqref{eq:cir} under the positivity-preserving schemes, with noncentral-$\chi^2$ CIR rows--exact for the frozen-$\theta$ weekly approximation--as a smoothing prior. (Appendix~\ref{app:chain}
documents why the naive analytic chain with bin-midpoint nodes is
inadmissible: it carries a $+34\%$ bias in annual water, a Jensen-type
artifact of right-skewed densities that a certification benchmark
catches immediately.)

The cost-to-go $V_{t,j}(s)$ is convex and nonincreasing in $s$
(Section~\ref{sec:model}) and is represented as the upper envelope of
affine minorants (``cuts''). The stage subproblem at $(t,j,s)$ is the
linear program
\begin{equation}
\begin{aligned}
V_{t,j}(s) = \min_{u,\,w,\,y,\,s',\,\varphi}\quad
& \Delta\Big(\textstyle\sum_{k} c^{\mathrm{seg}}_k y_k + \varepsilon_w w\Big)
 + \delta\,\varphi \\
\text{s.t.}\quad
& s' + \Delta u + \Delta w = s + \Delta q_{t,j}
&& \big[\mu\big] \\
& u + \textstyle\sum_k y_k \ \ge\ D_t, \qquad
0 \le y_k \le \bar y_k,\\
& 0 \le u \le U_{\max}, \qquad w \ge 0, \qquad
0 \le s' \le S_{\max},\\
& \varphi \ \ge\ a_m + b_m\, s' \qquad \forall m,
\end{aligned}
\label{eq:stageLP}
\end{equation}
where the $y_k$ implement the piecewise-linear thermal cost and
$(a_m,b_m)$ are the current cuts for the (risk-adjusted) future cost
$W_{t,j}(s') := \rho^{t,j}_\gamma\big(V_{t+1,\cdot}(s')\big)$. The dual
multiplier $\mu$ of the storage balance is exactly
$\partial V_{t,j}/\partial s$, i.e., $\ssv = -\mu$: \emph{the water value
is an LP dual}, and every SSV reported in Section~\ref{sec:results} is
read off HiGHS in this way (we verified $\mu$ against finite differences
to four decimals). Periodicity is imposed by wrapping stage $t=51$ to
$t=0$ cuts with the per-stage discount, the infinite-horizon periodic
construction of \citet{shapiro2020}; cut envelopes then increase
monotonically to the steady-seasonal fixed point.

The backward pass generates cuts at trial states visited by forward
simulation of the true SDE. Its correctness under the entropic risk
aggregation is the following lemma, which is where
Donsker--Varadhan does the work.

\begin{lemma}[validity of tilted cuts]\label{lem:cut}
Fix $t,j$ and a trial point $\hat s$. Suppose for each successor node
$j'$ we have an affine minorant
$v_{j'} + \beta_{j'}(\,\cdot - \hat s\,) \le V_{t+1,j'}(\cdot)$
with $v_{j'} = \underline V_{t+1,j'}(\hat s)$ (e.g.\ from solving
\eqref{eq:stageLP} at $(\,t{+}1,j',\hat s\,)$, with $\beta_{j'}$ its
storage dual). Let $\tilde w$ be the tilt \eqref{eq:tilt} evaluated at
$v$, and set
\[
\bar b \;=\; \sum_{j'} \tilde w_{j'}\, \beta_{j'},
\qquad
\bar a \;=\; \rho^{t,j}_\gamma(v) \;-\; \bar b\,\hat s .
\]
Then $\bar a + \bar b\, s \;\le\; W_{t,j}(s)$ for all
$s \in [0,S_{\max}]$, with equality at $s = \hat s$ whenever the
minorants are tight there. The statement covers the risk-neutral case
$\gamma \to 0$ with $\tilde w \to P_t[j,\cdot]$.
\end{lemma}

\begin{proof}
Fix $s$. By assumption and monotonicity of $\rho_\gamma$,
$W_{t,j}(s) = \rho_\gamma\big(V_{t+1,\cdot}(s)\big)
\ge \rho_\gamma\big(v + \beta(s-\hat s)\big)$.
By \eqref{eq:dv} with the (feasible, not necessarily optimal) measure
$\tilde w$,
\[
\rho_\gamma\big(v + \beta(s-\hat s)\big)
\;\ge\;
\E_{\tilde w}\big[v + \beta(s-\hat s)\big] - \tfrac1\gamma\KL{\tilde w}{P}
\;=\;
\Big(\E_{\tilde w}[v] - \tfrac1\gamma\KL{\tilde w}{P}\Big)
+ \bar b\,(s-\hat s).
\]
Since $\tilde w$ is the Donsker--Varadhan maximizer \emph{for the vector
$v$}, the bracket equals $\rho_\gamma(v)$, and the right-hand side is
$\bar a + \bar b s$. Equality at $\hat s$ is immediate.
\end{proof}

Lemma~\ref{lem:cut} says the natural construction--solve the successor
subproblems, tilt the probabilities by \eqref{eq:tilt}, average the
duals under the tilt--produces supporting hyperplanes of the
\emph{robust} cost-to-go. Cuts therefore remain valid lower bounds
throughout, and the lower bound reported by the algorithm is a true
bound on the risk-adjusted value. Convergence of sampled cutting-plane
schemes of this type (finitely many stages per cycle, finite noise
support, polyhedral subproblems, monotone valid cuts sampled along
forward trajectories) follows the standard arguments of
\citet{philpott2008,girardeau2015} extended to the periodic setting by
\citet{shapiro2020}; we do not restate them.

\subsection{InFlow algorithm and certification discipline}
\label{sec:anneal}

Algorithm~\ref{alg:inflow} states the certification-mode workflow. The key
implementation rule is that the guide and the certified cut store are distinct
objects.

\begin{algorithm}[t]
\caption{InFlow: periodic entropic SDDP with independent HJB cross-certification}
\label{alg:inflow}
\begin{algorithmic}[1]
\Require Seasonal model parameters; weekly stages $t=0,\ldots,51$; robustness
multiplier $\gamma\ge0$; annealing schedule $\eta_1>\cdots>\eta_K\downarrow0$.
\State Simulate the nonnegative seasonal CIR process and construct weekly-mean
nodes $q_{t,j}$ and transitions $P_t$; use frozen-$\theta$ CIR rows only as a
smoothing prior.
\State Independently solve the state-constrained periodic HJB problem on an
$(s,q)$ grid and store $\lambda_W^{\rm HJB}=-\partial_sV$ for cross-checking.
\State \textbf{Guide phase:} iterate the soft Bellman operator with nested
$\rho_\gamma$ along $\eta_1,\ldots,\eta_K$, followed by hard-min sweeps; retain
the resulting policy/value only as a state-exploration guide.
\State Initialize an empty certified cut store $\mathcal C$.
\For{SDDP iteration $k=1,2,\ldots$}
  \State Forward-simulate the true SDE; use the guide only to propose trial
  storage states (or the current exact policy after the first iteration).
  \For{$t=51,\ldots,0$ and each trial storage $\hat s$}
    \State Solve every successor stage LP with HiGHS using $\mathcal C$; record
    values $v_{j'}$ and storage-balance duals $\beta_{j'}$.
    \State Compute Gibbs weights
    $\tilde w_{j'}\propto P_t[j,j']e^{\gamma v_{j'}}$ and add the exact cut
    $(\bar a,\bar b)$ of Lemma~\ref{lem:cut} to $\mathcal C$.
  \EndFor
  \State Record the certified lower bound and dual water values
  $\lambda_W^{\rm SDDP}=-\mu$.
\EndFor
\State Cross-check $\lambda_W^{\rm SDDP}$ against the independent HJB gradient;
evaluate selected policies on common out-of-sample SDE trajectories.
\end{algorithmic}
\end{algorithm}

Phase A of the solver iterates the entropic operator
\eqref{eq:softlse}-- with $\rho_\gamma$ inside when $\gamma>0$ --on a
storage grid, along a temperature schedule
$\eta \in \{10^{-2}, 3\times10^{-3}, 10^{-3}\} \downarrow 0$, each stage
accelerated by one Richardson extrapolation along the annual
contraction (factor $e^{-\rho}$). By Proposition~\ref{prop:soft} the
final hard sweep is within grid error of the exact chain value. For the single-reservoir benchmark we also retain the original
\emph{benchmark mode}, in which conservative affine approximations extracted
from phase A initialize the working envelope before exact HiGHS refinement;
this is computationally useful for policy generation but it is not the source
of the paper's unconditional lower-bound claim. In \emph{certification mode}
(Algorithm~\ref{alg:inflow}) the exact store starts empty and phase A affects
only trial-state selection.

The status of those seeds requires care, and we state it precisely because
it is easy to overclaim. Proposition~\ref{prop:soft}(i) says the entropic
operator \emph{over}estimates: $\mathcal T\le\mathcal T^\eta$. Affine
minorants of a phase-A value function are therefore minorants of $V^\eta$
and not, a priori, of the exact chain value; the discrepancy is bounded by
$\eta\log N_u/(1-\delta)$ plus phase-A discretization error, and at weekly
discounting the first term alone is numerically uninformative
(Appendix~\ref{app:softproof}). Seeds are consequently \emph{not}
guaranteed valid cuts, and the correct discipline is a strict separation
between guidance and certification: the entropic phase may determine where
the exact solver generates cuts, never what the solver certifies. Under
that discipline the reported lower bound is valid unconditionally by
Lemma~\ref{lem:cut} and Theorem~\ref{thm:posterior}. Where a seeded working envelope is used for the single-reservoir policy
experiments, it is labelled explicitly as benchmark mode. The cross-certification
claim is additionally repeated with an exact-cut-only store, and the cascade
uses strict guide/store separation throughout.

In one storage dimension a grid is affordable and phase A could stand
alone. The architecture was designed on the expectation that in higher
dimensions phase A would be replaced by soft \emph{cuts}--the log-sum-exp
of affine functions is smooth and convex--with the identical annealing
role, giving a smooth global phase that finds the region and an exact
polyhedral phase that certifies it. Section~\ref{sec:cascade} tests that
expectation on a two-reservoir cascade and does not confirm it: the
certification layer transfers, the acceleration does not. We retain the
construction as a proposal and mark it explicitly as unvalidated.

\subsection{Numerical protocol and reproducibility}
\label{sec:protocol}

The normalized benchmark parameters are summarized in Table~\ref{tab:params}.
The values are scenario-design parameters chosen to create a seasonal storage
problem with a pronounced dry-season stress; they are not fitted estimates for
an identified reservoir.

\begin{table}[t]
\centering
\caption{Normalized single-reservoir benchmark parameters.}
\label{tab:params}
\begin{tabular}{lll}
\toprule
Quantity & Value & Interpretation\\
\midrule
$S_{\max}$ & $0.4$ & storage capacity (about five months of mean inflow)\\
$U_{\max}$ & $3.0$ & maximum release rate\\
$\kappa,\sigma_0$ & $8,\,2$ & CIR mean reversion and volatility\\
$\bar\theta,a$ & $1,\,0.8$ & seasonal inflow mean and amplitude\\
$\bar D,b$ & $1,\,0.4$ & demand mean and amplitude\\
$c_1,c_2$ & $0.5,\,2$ & thermal shortfall-cost coefficients\\
$\varepsilon_w,\rho$ & $0.05,\,0.1$ & spill penalty and annual discount rate\\
Stages / nodes / segments & $52/11/8$ & weekly stages, inflow nodes, PWL cost segments\\
\bottomrule
\end{tabular}
\end{table} Mean-reverting square-root dynamics with a periodic level are a
standard reduced-form representation of seasonal inflow and of
mean-reverting energy-market drivers \citep{cir1985,schwartz1997}. Pronounced
annual seasonality is also observed empirically in long-term Senegal River
hydropower inflow records \citep{affognon2026wavelet}; the normalized CIR
benchmark used here is deliberately not calibrated to those observations. A
storage capacity of a few months of mean inflow with counter-seasonal
demand is the configuration in which seasonal transfer, rather than
within-week balancing, determines the water value
\citep{gjelsvik1992,loucks2017}. The Feller ratio $4\theta(t)$ falls below
one in the five driest weeks (minimum $0.8$); Figure~\ref{fig:validation} (top) shows the
seasonal geometry. The baseline HJB benchmark uses a $41\times41$ $(s,q)$ grid with $2080$
steps per year and an explicit monotone upwind scheme; refinement studies use
$21^2$, $61^2$, and $81^2$ state grids with time steps tightened to maintain
the CFL condition. All stage LPs are solved by HiGHS \citep{huangfu2018}.
The Markov chain has $11$ weekly-mean inflow nodes and the stage LP uses eight
piecewise-linear thermal-cost segments. The main out-of-sample comparison uses
$3{,}000$ independent trajectories, four simulated years per trajectory with
the first year discarded as warmup, and common random numbers across
$\gamma\in\{0,2,5\}$. Policy contrasts are bootstrapped by trajectory, not by
year, with $10{,}000$ paired resamples. All random seeds and scripts are
included in the accompanying reproducibility archive.

Because the HJB benchmark has a continuous inflow coordinate while the
weekly chain is node-based, scalar cross-comparisons use an explicitly
matched pair of reference states,
\begin{equation}
\begin{aligned}
 x_{\mathrm{ref}}^{\mathrm{HJB}}&=(t=0,\;s=0.2000,\;q=1.0000),\\
 x_{\mathrm{ref}}^{\mathrm{chain}}&=(t=0,\;s=0.2000,\;j=5,\;q_{0,5}=1.0923).
\end{aligned}
\label{eq:refstate}
\end{equation}
where the continuous state $q=1$ (the annual mean) maps to chain node $5$
of $11$, whose conditional mean weekly inflow is $1.0923$. Thus the two
states represent the same intended physical reference--the first week of
the annual cycle and a half-full reservoir--without identifying the
continuous inflow value with the node conditional mean. The risk-adjusted
values $0.657$, $0.775$ and $1.005$ quoted for $\gamma=0,2,5$ in
Section~\ref{sec:pessimistic}, and the chain dynamic-program value $0.6619$
in Section~\ref{sec:certification}, are evaluated at
$x_{\mathrm{ref}}^{\mathrm{chain}}$; the HJB reference state is used only for diagnostic value-level comparisons;
the primary cross-certification object is the seasonal storage derivative. Seasonal SSV profiles are instead
reported as functions of week $t$ at a fixed storage fraction and at the
inflow node nearest the seasonal median, $j(t)=\arg\min_j
|q_{t,j}-\theta(t)|$, because a single fixed node would be atypical in
most of the year; the storage fraction is stated with each figure.

\section{Results}
\label{sec:results}

\begin{figure}[t]
\centering
\includegraphics[width=.86\linewidth]{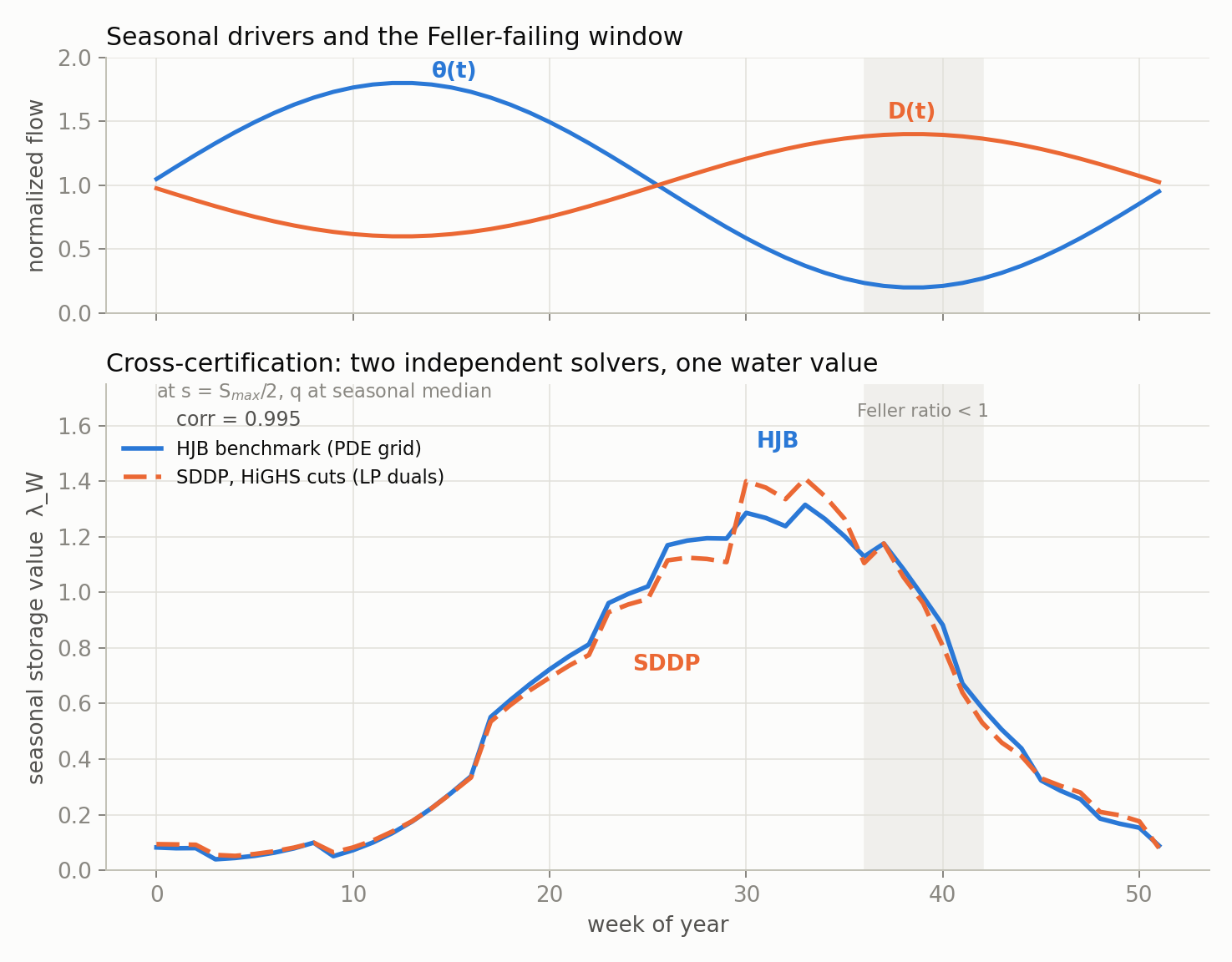}
\caption{Top: seasonal inflow mean $\theta(t)$ and demand $D(t)$; the
shaded band marks the weeks with Feller ratio $<1$. Bottom:
cross-certification of the SSV. The water value computed as the storage
gradient of the HJB viscosity solution (blue) and as the HiGHS dual of
the SDDP storage balance (orange, dashed) agree with correlation
$0.995$; both peak entering the drought window and collapse in the wet
filling season.}
\label{fig:validation}
\end{figure}

\subsection{HJB--SDDP cross-certification and numerical refinement}
\label{sec:certification}

Figure~\ref{fig:validation} (bottom) is the central cross-check: two solvers
with disjoint machinery--a PDE storage gradient and a linear-programming
balance dual--produce seasonal storage values with correlation $0.995$ and a
mean level difference of only a few percent. Both peak on the approach to the
dry window and collapse during refill. The comparison first rejected a
seemingly reasonable midpoint chain that overstated annual inflow by $34\%$;
replacing midpoint representatives by simulated conditional means reduced the
annual-mean inflow discrepancy to $0.5\%$. This is precisely the type of model
representation error that an internal SDDP convergence statistic cannot
detect.

Two additional checks strengthen the comparison. First, an exact-cut-only
SDDP replication, in which the certified store contains no phase-A-derived
seed cuts, gives correlation $0.9960$ with the baseline HJB SSV profile, with a
52-week HJB--SDDP seasonal water-value profile RMSE of $0.0425$ and an
annual-mean level difference of $-0.0105$. Its scalar lower
bound is still increasing after the reported sweeps, so we use this experiment
as a \emph{gradient cross-check}, not as a claim that cold-start SDDP has
converged in value. Second, HJB refinement shows that the water-value derivative
stabilizes appreciably faster than the absolute value level. Table~\ref{tab:hjbmesh}
reports the refinement sequence. From $61^2$ to $81^2$ states the annual-mean
SSV changes from $0.5783$ to $0.5747$ ($0.6\%$) and the peak changes from
$1.3332$ to $1.3344$ (less than $0.1\%$), whereas the absolute value at the
reference state still changes materially. Increasing the inflow truncation
from $q_{\max}=4.5$ to $6.0$ at comparable resolution leaves the annual-mean
SSV essentially unchanged ($0.5783$ in both runs). We therefore base the
cross-certification claim on the \emph{marginal water value}, and we do not use
the coarse-grid HJB value level to decompose a continuous/discrete error.

\begin{table}[t]
\centering
\caption{HJB refinement diagnostic. $V_{\rm ref}$ is shown to document its
slower grid convergence; the SSV derivative is the cross-certification object.}
\label{tab:hjbmesh}
\begin{tabular}{rrrrr}
\toprule
$(N_s,N_q)$ & steps/year & $V_{\rm ref}$ & mean SSV & peak SSV\\
\midrule
$(21,21)$ & 1040 & 1.0430 & 0.6051 & 1.3305\\
$(41,41)$ & 2080 & 0.8756 & 0.5852 & 1.3319\\
$(61,61)$ & 4680 & 0.8200 & 0.5783 & 1.3332\\
$(81,81)$ & 8320 & 0.7916 & 0.5747 & 1.3344\\
\bottomrule
\end{tabular}
\end{table}

\begin{figure}[t]
\centering
\includegraphics[width=.8\linewidth]{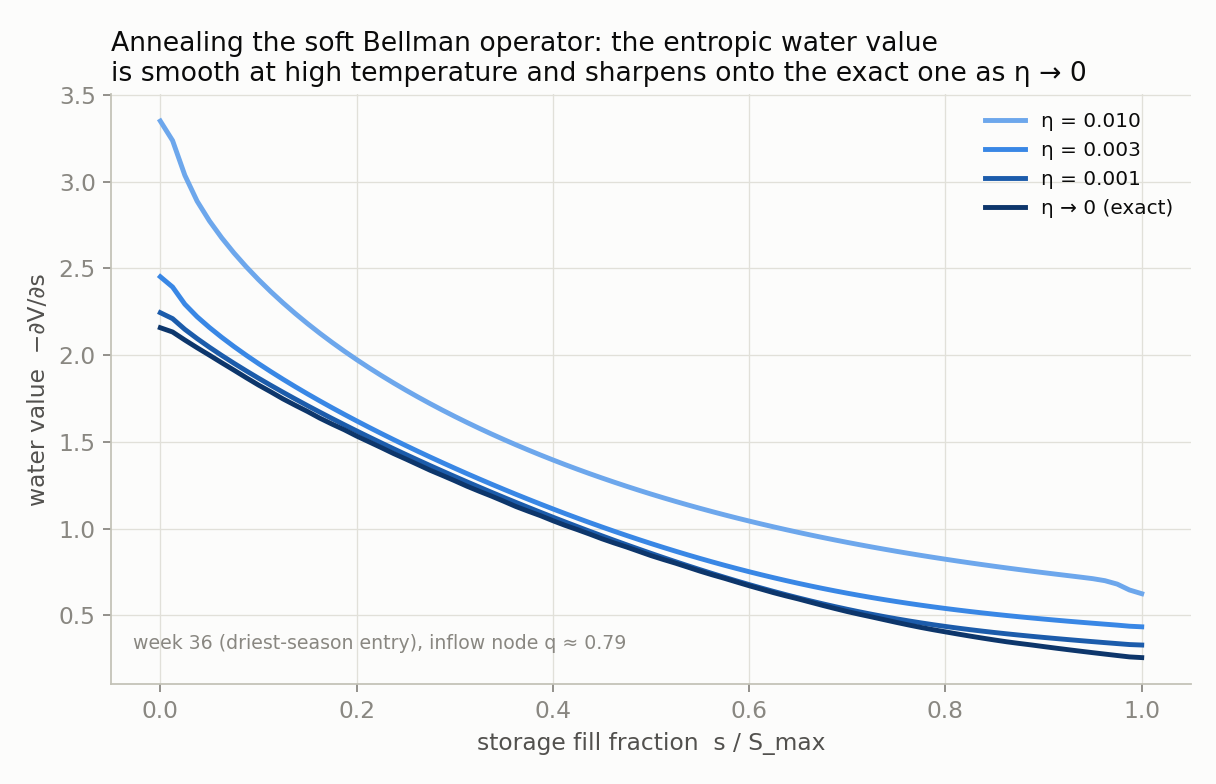}
\caption{Annealing the entropic Bellman operator
(Proposition~\ref{prop:soft}) at the entry to the dry season: the water
value $-\partial_s V^\eta$ is smooth and biased upward at high
temperature and sharpens monotonically onto the exact profile as
$\eta \downarrow 0$.}
\label{fig:annealing}
\end{figure}

\subsection{Entropic annealing and the robust seasonal storage value}
\label{sec:pessimistic}

Figure~\ref{fig:robust} shows the SSV under
$\gamma \in \{0, 2, 5\}$. These three values were fixed before the
out-of-sample comparison as transparent design points: $\gamma=0$ is the
risk-neutral limit, $\gamma=2$ represents moderate model distrust, and
$\gamma=5$ is a deliberately strong stress setting. They are a sensitivity
bracket, not estimates and not evidence of global monotonicity of the SSV
gradient for every $\gamma$. For the three tested policies, the robust water
value dominates the neutral one everywhere, with risk-adjusted values at
$x_{\mathrm{ref}}^{\mathrm{chain}}$ of $0.657$, $0.775$, and $1.005$, respectively.
The entropic cost functional itself is continuous and nondecreasing in
$\gamma$ for a fixed random cost, but differentiation with respect to
storage means that pointwise monotonicity of its gradient is not automatic.

In an operational calibration, $\gamma$ should be selected jointly with an
empirical KL ambiguity radius rather than chosen from these three labels.
One practical procedure is to estimate seasonal transition rows on rolling
or bootstrap inflow samples, construct a confidence radius $r$ for their KL
deviation, and choose the multiplier $\gamma$ whose exponentially tilted
transition rows attain the target radius (or choose on the validated
mean--CVaR frontier subject to an operator's reliability constraint).
Stability should then be reported over confidence levels and time blocks.
This is the empirical counterpart of the radius--multiplier relation in
Eq.~\eqref{eq:entropicrisk}; the present synthetic experiment demonstrates the
mechanism but does not perform that statistical calibration. The structural
finding is the
\emph{location} of the lift: in absolute terms it is largest in the
storage-building weeks $15$--$30$ and $42$--$48$, and proportionally it
is smallest inside the drought window itself. The economics are
transparent ex post. Within the drought, scarcity is realized and the
SSV is pinned by current shortage costs, which the adversary cannot
worsen much; ahead of the drought, the adversary can still deepen the
coming dry season, and the only hedge is to carry more storage into it.
Robustness therefore operates as a \emph{precautionary buffer
instrument}: it raises the price of water exactly when there is still
time to store it. This shoulder-season concentration is therefore treated as a structural
finding of the present benchmark, not as a universal law for all storage
systems.

\begin{figure}[t]
\centering
\includegraphics[width=.86\linewidth]{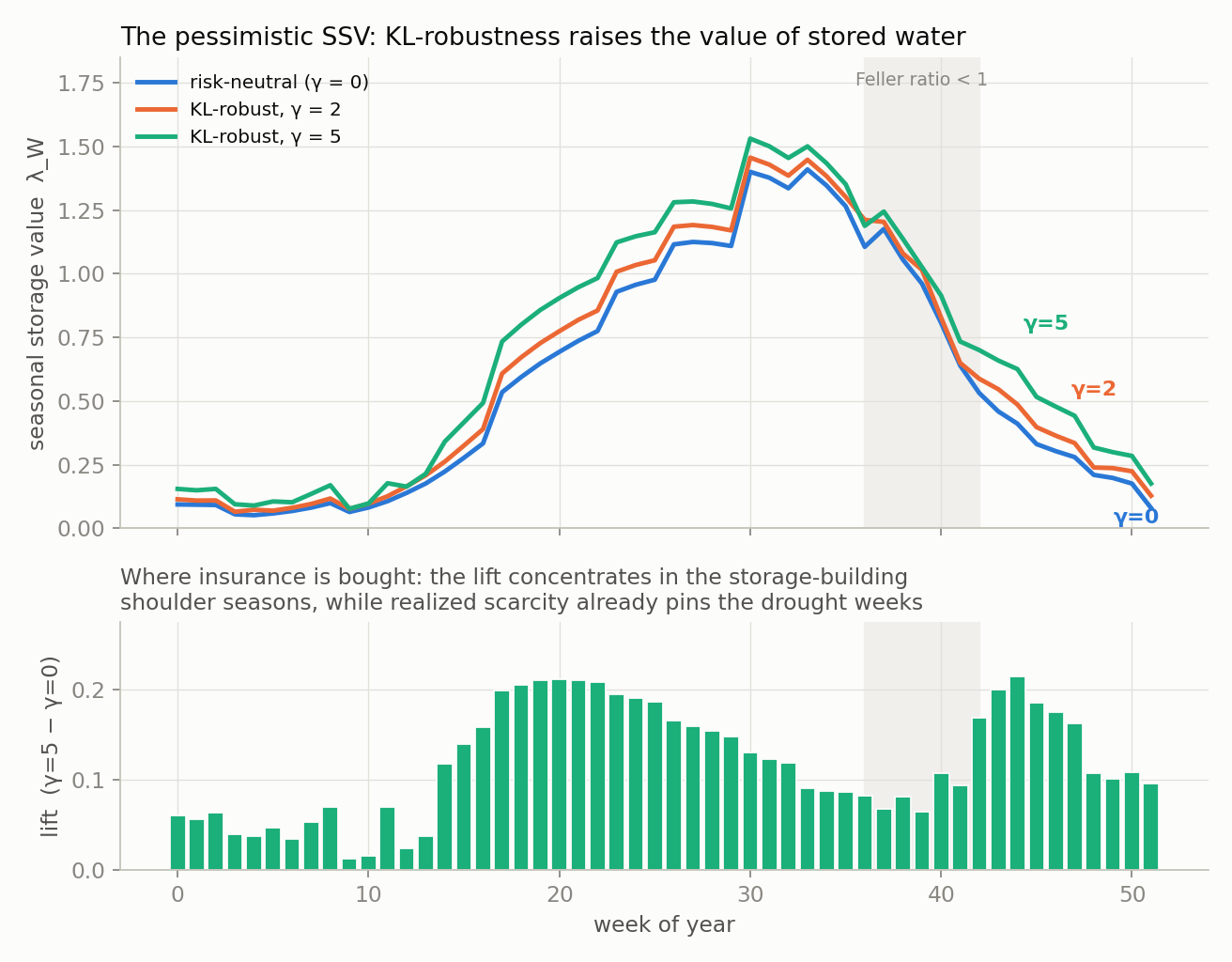}
\caption{Top: seasonal storage value at mid-storage under
$\gamma\in\{0,2,5\}$; the robust SSV dominates the neutral one.
Bottom: the lift $\gamma{=}5$ minus $\gamma{=}0$ concentrates in the
shoulder seasons where hedging by storage is still possible, not in the
drought where scarcity is already realized.}
\label{fig:robust}
\end{figure}

\subsection{Out-of-sample value of robustness}
\label{sec:oos}

The benchmark-mode SDDP policies were evaluated on the original continuous
SDE using $3{,}000$ independent trajectories. Each trajectory is simulated for
four years, with the first discarded as warmup, yielding $9{,}000$ evaluation
trajectory-years. The same random inflow paths are used for all policies. The
online decision objective uses the same eight-segment piecewise-linear stage
cost as the HiGHS subproblem, whereas realized performance is scored using the
original quadratic thermal cost in \eqref{eq:cost}. An exact breakpoint
enumeration of the one-dimensional convex stage objective was checked against
$2{,}000$ random HiGHS stage solves and reproduced every release to numerical
precision. Two worlds are considered: the nominal calibration and a stressed
world in which the dry-season mean inflow is $40\%$ lower. Uncertainty in
policy contrasts is quantified by a paired trajectory-cluster bootstrap with
$10{,}000$ resamples.

\begin{table}[t]
\centering
\caption{Out-of-sample yearly cost from $3{,}000$ independent trajectories.
CVaR$_{90}$ is the mean of the worst decile. Percentages are relative to
$\gamma=0$ in the same world.}
\label{tab:oos}
\begin{tabular}{llccc}
\toprule
World & Policy & mean & CVaR$_{90}$ & worst year \\
\midrule
nominal & $\gamma=0$ & 0.07490 & 0.31947 & 0.6901\\
nominal & $\gamma=2$ & 0.07505 ($+0.2\%$) & 0.30617 ($-4.2\%$) & 0.6736\\
nominal & $\gamma=5$ & 0.07833 ($+4.6\%$) & 0.28355 ($-11.2\%$) & 0.6433\\
\midrule
stressed & $\gamma=0$ & 0.10446 & 0.39058 & 0.7358\\
stressed & $\gamma=2$ & 0.10321 ($-1.2\%$) & 0.37476 ($-4.0\%$) & 0.7161\\
stressed & $\gamma=5$ & 0.10406 ($-0.4\%$) & 0.34562 ($-11.5\%$) & 0.6909\\
\bottomrule
\end{tabular}
\end{table}

For $\gamma=5$ relative to $\gamma=0$, the nominal mean-cost difference is
$0.00343$ with paired-bootstrap $95\%$ interval $[0.00306,0.00381]$, while
the CVaR$_{90}$ difference is $-0.03592$ with interval
$[-0.03773,-0.03409]$. Under stress, the mean difference is $-0.00039$ with
interval $[-0.00084,0.00005]$, so it is not resolved from zero, whereas the
CVaR$_{90}$ difference is $-0.04495$ with interval
$[-0.04641,-0.04350]$. The interpretation is therefore stable under the
larger experiment: robustness behaves like insurance in the nominal world,
trading a modest mean premium for a substantially thinner bad-year tail;
when the dry-season model is wrong in the pessimistic direction, the premium
vanishes while the tail protection remains.

\begin{figure}[t]
\centering
\includegraphics[width=.9\linewidth]{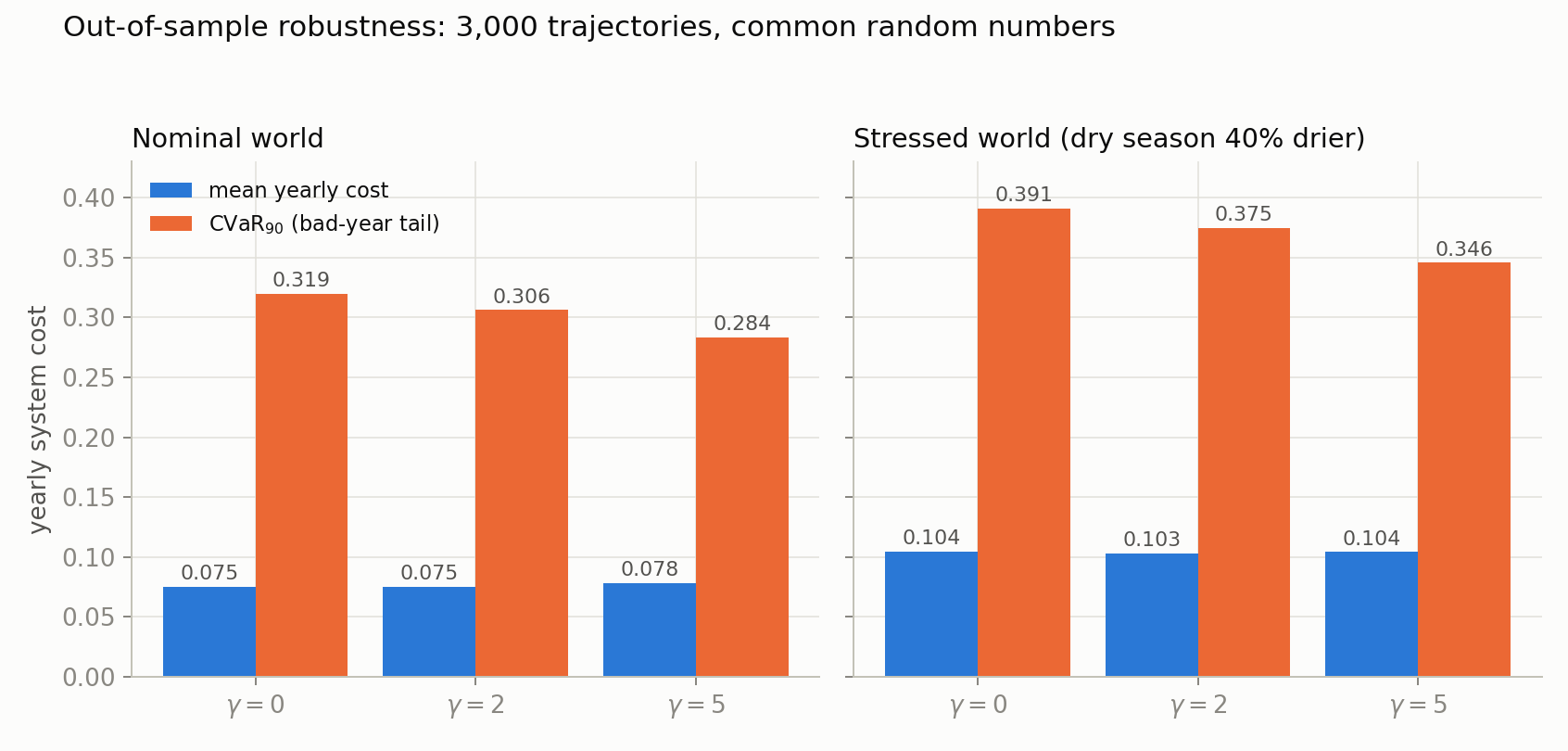}
\caption{Out-of-sample mean and CVaR$_{90}$ of yearly cost under the nominal
and stressed worlds, using $3{,}000$ independent trajectories and common
random numbers across policies.}
\label{fig:oos}
\end{figure}

\subsection{Sensitivity to the robustness coefficient}
\label{sec:gammasweep}

The three values $\gamma\in\{0,2,5\}$ reported above were chosen as a
risk-neutral baseline, a mild hedge and a strong hedge; they are not
distinguished points of the model. To establish that nothing in the
comparative statics depends on that choice, the value/dual pipeline--annealed phase A, cut generation, and exact
refinement--was re-run on the twelve-point grid
$\gamma\in\{0,\,0.25,\,0.5,\,1,\,1.5,\,2,\,3,\,4,\,5,\,6.5,\,8,\,10\}$.
Figure~\ref{fig:gammasweep} reports the result.

The risk-adjusted value at $x_{\mathrm{ref}}^{\mathrm{chain}}$ rises smoothly and strictly
from $0.657$ at $\gamma=0$ to $1.518$ at $\gamma=10$, with forward
differences increasing monotonically from $0.053$ to $0.110$: the curve is
strictly increasing and numerically convex, with no kink, plateau or
threshold anywhere in the range. Strict monotonicity is not an artifact but
a theorem--Proposition~\ref{prop:Vmono} in Appendix~\ref{app:gamma}, a
consequence of $\gamma\mapsto\rho_\gamma$ being nondecreasing--and the
computation confirms it at all $11$ successive gaps. The
\emph{season-averaged} seasonal storage value is likewise strictly
increasing, from $0.582$ to $0.840$. The qualitative finding of
Section~\ref{sec:pessimistic}--that the robust uplift concentrates in the
storage-building shoulder seasons rather than inside the drought--is
stable across the whole range, as the right panel of
Figure~\ref{fig:gammasweep} shows: the shape of the lift is essentially
invariant and only its amplitude scales with $\gamma$. The three reported
values are therefore representative, and $\gamma=5$ is not a special point.

Two caveats belong in the record. First, monotonicity holds for the value
and for the season-averaged SSV, but \emph{not} pointwise for the weekly
SSV: across the $11\times52$ successive increments, $49$ ($8.6\%$) are
negative, with mean magnitude $0.026$, i.e.\ about $6\%$ of the local
level. This is expected rather than anomalous. Proposition~\ref{prop:Vmono}
orders the value functions, and a pointwise-increasing family of convex
functions need not have ordered slopes (Remark~\ref{rem:ssvmono}). A
control experiment quantifies how much of the effect is numerical: re-running
$\gamma=2$ with five different random seeds gives a seed-to-seed
peak-to-peak spread of the weekly dual of $0.0094$ on average, and
$53\%$ of the observed reversals are smaller than that spread. The
remainder is genuine non-monotonicity of the gradient. We therefore claim
monotonicity of the value (proved) and of the season-averaged SSV
(observed), and explicitly not pointwise monotonicity of the weekly SSV.

Second, the small-$\gamma$ behaviour admits an independent analytic check.
Proposition~\ref{prop:rhomono} gives
$\rho_\gamma(v)=\E[v]+\tfrac{\gamma}{2}\operatorname{Var}(v)+O(\gamma^2)$,
so by the envelope theorem the initial slope of the value in $\gamma$ must
equal one half of the discounted sum of the conditional variances of the
continuation value along the nominal closed loop. Evaluated inside the
chain dynamic program--where both sides use the identical operator--the
secants $(V^{(\gamma)}-V^{(0)})/\gamma$ at
$\gamma\in\{0.05,0.1,0.2,0.4\}$ extrapolate to $0.0526$ as
$\gamma\downarrow0$, against an analytic prediction of
$0.0489\pm0.0001$, an agreement of $7.7\%$ attributable to the storage
grid and to the neglected $O(\gamma)$ term. The same comparison performed
on the end-to-end pipeline value instead of the chain value disagrees by a
factor $1.5$, because the pipeline additionally contains the
piecewise-linear cost surrogate and the sampled cut envelope; we report the
chain-internal comparison because it is the only one that isolates the
quantity the expansion actually describes.

\begin{figure}[t]
\centering
\includegraphics[width=\linewidth]{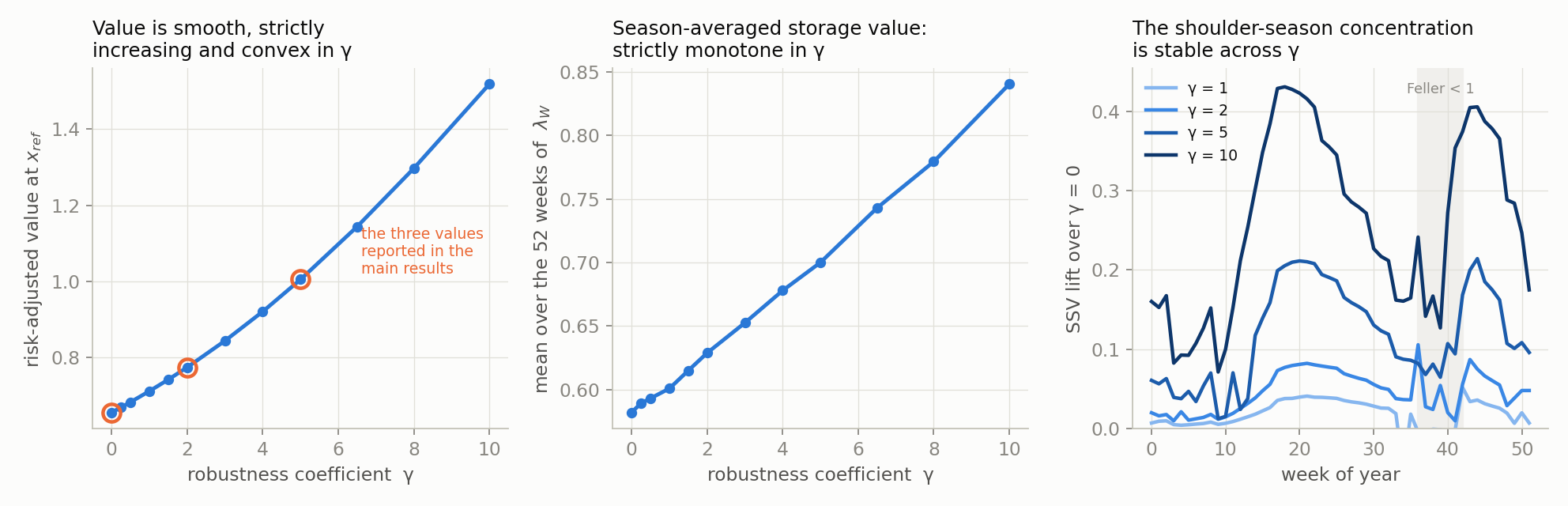}
\caption{Comparative statics in the robustness coefficient over a
twelve-point grid. Left: the risk-adjusted value at the reference state
\eqref{eq:refstate} is smooth, strictly increasing (as
Proposition~\ref{prop:Vmono} requires) and numerically convex; the circled
points are the three values reported in the main results. Centre: the
season-averaged seasonal storage value is likewise strictly monotone.
Right: the shoulder-season concentration of the robust uplift is invariant
in shape across $\gamma$ and scales only in amplitude.}
\label{fig:gammasweep}
\end{figure}

\subsection{Information timing: an exactly paired chain experiment}
\label{sec:timing}

The weekly benchmark uses a decision--hazard convention: the current weekly
inflow node $j_t$ is observed before release $u_t$ is selected. An operational
implementation may instead require a hazard--decision convention in which the
decision is selected using only $j_{t-1}$ and must remain feasible for every
current-node realization. To isolate the value of this information, both
conventions are solved on the \emph{identical} chain--same nodes, transition
matrices, costs, and storage grid. Their difference therefore contains no
change in state or temporal discretization.

The pointwise information advantage is
\begin{equation}
\Delta_{\mathrm{info}}(t,s,i)=V^{\mathrm{HD}}_t(s,i)-
\sum_jP_{t-1}[i,j]V^{\mathrm{DH}}_t(s,j),
\label{eq:definfo}
\end{equation}
which Proposition~\ref{prop:info} proves nonnegative. At the reference chain
state, $\Delta_{\mathrm{info}}=0.0458$. Over the complete
$52\times11\times81$ grid its mean is $0.0489$, median $0.0468$, $95$th
percentile $0.0589$, and maximum $0.0769$; the minimum remains positive
($0.0447$). At half-full storage the seasonal average over previous nodes
ranges only from $0.0456$ to $0.0519$, peaking in week 29. Because the HJB
absolute value level is more mesh-sensitive than its derivative
(Table~\ref{tab:hjbmesh}), we deliberately do \emph{not} express
$\Delta_{\mathrm{info}}$ as a percentage decomposition of the HJB--chain
value-level gap.

The timing distinction is operationally non-negligible. Simulating both grid
policies on the continuous SDE gives mean yearly costs $0.07318$ (DH) versus
$0.07740$ (HD) in the nominal world, a $5.8\%$ penalty for the honest timing
convention; under stress the corresponding costs are $0.10025$ and $0.10481$,
a $4.6\%$ penalty. These comparisons concern the information convention only
and do not alter the robustness ranking, because all robustness policies are
compared under a common timing convention.

\begin{figure}[t]
\centering
\includegraphics[width=\linewidth]{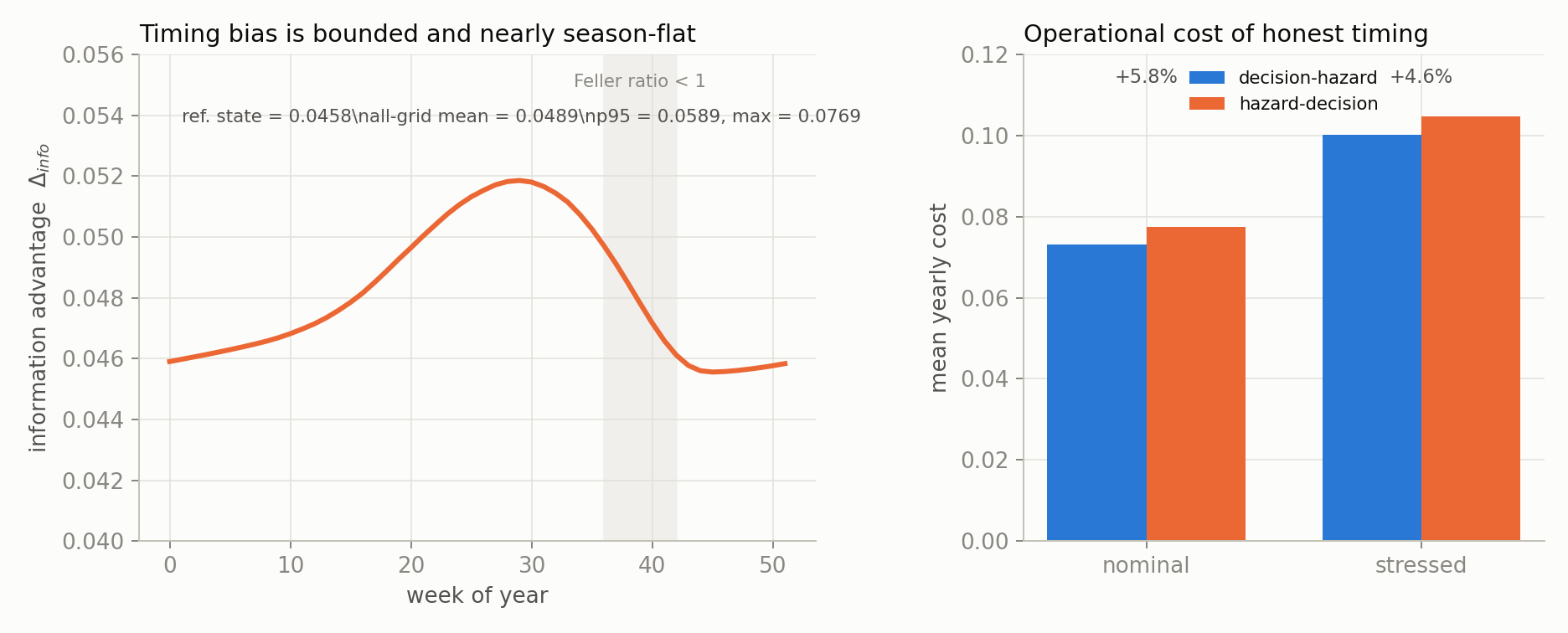}
\caption{Paired timing experiment on the identical weekly chain. Left: the
nonnegative information advantage $\Delta_{\rm info}$ is bounded and nearly
season-flat at half-full storage. Right: mean yearly cost of decision--hazard
and hazard--decision policies when both are simulated on the continuous SDE.}
\label{fig:timing}
\end{figure}

\subsection{A two-reservoir cascade}
\label{sec:cascade}

A two-reservoir toy cascade tests whether the cross-certification machinery
and the proposed higher-dimensional soft-cut extension survive beyond one
storage dimension. The experiment produces one positive and one negative
result.

The instance places two reservoirs in series on a shared seasonal driver: a
fraction $\alpha=0.6$ of the basin drains into the upper reservoir
($S_1^{\max}=0.25$), the remainder laterally into the lower one
($S_2^{\max}=0.20$), and water released or spilled upstream becomes inflow
downstream. Generation is $u_1+u_2$, so a unit entering the upper reservoir
passes both turbines while a unit entering laterally passes only the lower
one; mean energy availability is accordingly
$\alpha\cdot2+(1-\alpha)\cdot1=1.6$ per unit of mean inflow, and demand is
rescaled by $1.7$ so that the cascade is comparably stressed rather than
trivially over-supplied. The state is $(s_1,s_2,q)$, still small enough for
a grid dynamic program to serve as an oracle--the cascade analogue of the
HJB benchmark--at $25\times25$ storage points and a $17\times17$ release
grid.

\subsubsection*{Certification transfers} Cuts of the form
$a+b_1s_1'+b_2s_2'$, with $(b_1,b_2)$ read from the two storage-balance
duals of the stage linear program, reproduce \emph{both} water values of the
oracle across the annual cycle with correlation $0.9987$ upstream and
$0.9986$ downstream and a mean relative level gap of $6.2\%$ and $6.0\%$
respectively (Figure~\ref{fig:cascade}, left). The cross-certification
design is therefore not an artifact of one storage dimension. As an
independent physical check, the ratio $\lambda_1/\lambda_2$ of upstream to
downstream water value has annual mean $2.006$ and varies between $1.99$
and $2.02$ across all $52$ weeks (Figure~\ref{fig:cascade}, right),
matching the elementary prediction that upstream water is worth twice
downstream water because it generates twice. The solver recovers this
without being told.

\subsubsection*{The soft-cut acceleration claim does not survive testing} The
proposed role of the entropic phase in higher dimensions was to accelerate
the exact phase. Testing it required first correcting a design flaw. By
Proposition~\ref{prop:soft}(i) the entropic operator \emph{over}estimates,
so affine minorants extracted from a phase-A value function are minorants
of $V^\eta$, not of the exact value, and using them as cuts invalidates the
lower bound (Remark~\ref{rem:seedvalidity}). In the cascade this is
quantitatively visible: a converged coarse annealed phase A on a
$13\times13$ grid evaluates to $1.9866$ at the reference state against an
oracle value of $1.9119$, an overestimate of $3.9\%$. We therefore adopt a
strict separation--the entropic phase may determine \emph{where} the exact
solver looks, never \emph{what} it certifies--so that every reported bound
comes from exact HiGHS cuts and is valid by Lemma~\ref{lem:cut} and
Theorem~\ref{thm:posterior} regardless of the quality of the entropic
phase.

Under that rigorous protocol, guidance does not help. After $18$
iterations the exact-only lower bound is $0.448$ from a cold start and
$0.385$ when the forward pass is guided by the annealed phase; the
policies induced by the two exact cut families cost $0.773\pm0.082$ and
$0.848\pm0.083$ respectively on the true dynamics, a difference within one
standard error. We therefore withdraw the acceleration claim: in this
instance the annealed entropic phase confers no measurable advantage on the
exact cutting-plane phase, and Section~\ref{sec:anneal}'s
higher-dimensional soft-cut construction remains an architectural proposal
without numerical support. What the cascade does establish is the
certification claim in two storage dimensions and the correctness of the
two-dual cut construction.

The reader should note what this reclassification costs and what it does not.
The main single-reservoir policy experiments use the seeded working envelope
because cold-start exact SDDP is deliberately slow on the 52-stage periodic
benchmark. Those policy experiments are therefore reported as numerical
performance comparisons, not as lower-bound certificates. The certification
claim rests instead on Lemma~\ref{lem:cut}, Theorem~\ref{thm:posterior}, the
strict guide/store construction in Algorithm~\ref{alg:inflow}, and the
exact-cut-only gradient replication reported in Section~\ref{sec:certification}.
This separation prevents a useful numerical warm start from being confused
with a mathematical certificate.

\begin{figure}[t]
\centering
\includegraphics[width=\linewidth]{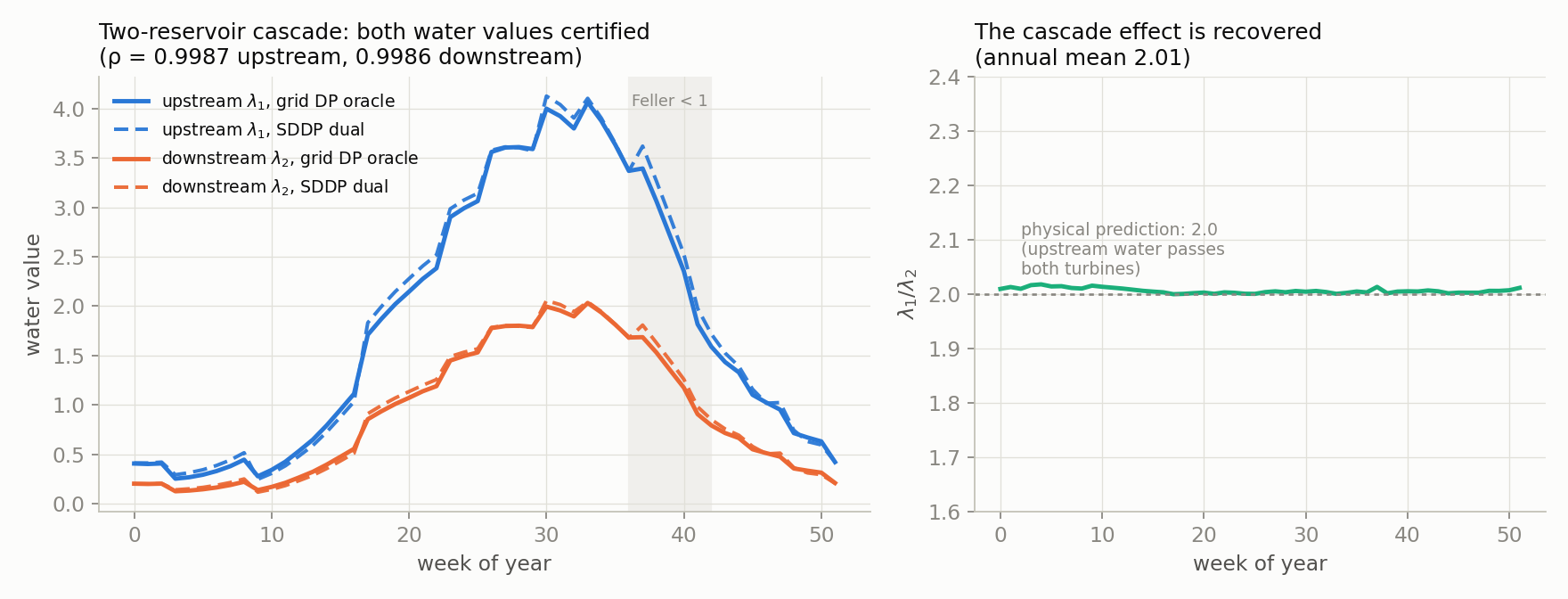}
\caption{Two-reservoir cascade. Left: upstream and downstream water values
from the grid dynamic-programming oracle (solid) and from the two
storage-balance duals of the stage linear program (dashed); correlations
$0.9987$ and $0.9986$ over the annual cycle. Right: the ratio
$\lambda_1/\lambda_2$ has annual mean $2.006$, recovering without
instruction the physical prediction that upstream water generates twice.}
\label{fig:cascade}
\end{figure}

\section{Discussion}
\label{sec:discussion}
\label{sec:domains}

The central result is methodological: the marginal value of stored water can
be computed by two independent numerical routes--as $-\partial_sV$ from a
state-constrained HJB equation and as the negative storage-balance dual from a
periodic stage LP--and the two seasonal profiles agree closely. The HJB
refinement study is important for interpreting that agreement. The derivative
stabilizes substantially faster than the absolute value level, which is why
the paper treats water-value cross-certification as the primary numerical
claim and avoids converting the coarse-grid value discrepancy into a formal
continuous/discrete error bound. In turn, the exact-cut-only SDDP replication
shows that the gradient agreement is not created by entropic seed cuts.

The robustness result has a transparent storage interpretation. The nested
entropic operator reweights successor nodes toward expensive continuation
states. Ahead of the dry season, the controller still has an intertemporal
hedge--retain water--so model distrust raises the marginal value of storage.
Once drought scarcity is realized, current shortage costs already dominate and
the marginal effect of further pessimism is smaller. This explains the
shoulder-season concentration in Figure~\ref{fig:robust}. The larger
out-of-sample experiment supports the economic interpretation: at
$\gamma=5$ the nominal policy pays a $4.6\%$ mean premium for an $11.2\%$
CVaR$_{90}$ reduction, whereas under the deliberately drier world the mean
difference is unresolved from zero and the tail reduction remains $11.5\%$.
These are benchmark-specific magnitudes, not universal hydropower constants.

The distinction between multiplier robustness and radius-constrained DRO is
also operationally important. InFlow solves the KL-penalized recursion at a
chosen $\gamma$. A field application should not interpret $\gamma$ as a
confidence radius. Instead, seasonal transition rows can be estimated on
bootstrap or rolling samples, an empirical KL radius can be constructed, and
the multiplier can then be selected through the radius--multiplier dual
relation or by validation on a reliability frontier \citep{lam2019,duchi2021}.
The dense $\gamma$ sweep shows that the risk-adjusted value is ordered and
smooth over the investigated range, but it does not eliminate the need for
statistical calibration.

The benchmark deliberately exposes rather than hides three limitations.
First, the one-reservoir model omits head effects, travel times, environmental
flows, unit commitment, and market detail. The two-reservoir cascade verifies
the dual construction in two storage dimensions and recovers the expected
upstream/downstream value ratio, but remains a toy. Second, the weekly
information convention has measurable value; the paired chain experiment
bounds it directly and shows a roughly five-percent policy-cost effect.
Third, control-space annealing is useful as a smooth homotopy and exploration
guide, but the cascade provides no evidence that it accelerates exact SDDP in
higher dimension. This negative result is why Algorithm~\ref{alg:inflow}
separates guidance from certification.

The numerical study is dimensionless by design. Flow is normalized by mean
annual inflow, storage by the corresponding annual volume, and cost by a
reference marginal thermal cost. The reported values are therefore algorithmic
and comparative, not forecasts, operating recommendations, or monetary water
values for an identified reservoir. This separation is consistent with the
recent water-value literature, which emphasizes uncertainty treatment,
physical fidelity, and transparent reporting of model assumptions
\citep{pavicevic2026water}.

Application to the Senegal river basin or another operational system is
therefore a subsequent validation stage, not an implicit claim of this
paper. That stage should calibrate plant-specific inflows, represent the
reservoir network and operating constraints, compare against the
incumbent planning workflow, and report physical and monetary water
values. For the planned MOSSHOOS application, the sequence is calibration
and validation at Manantali, followed by an OMVS cascade model. The present
contribution supplies the verified algorithmic component required before
that higher-dimensional application and is not redundant with it.

\section{Conclusion}
\label{sec:conclusion}

InFlow combines periodic stochastic dual dynamic programming, nested entropic
transition robustness, control-space Bellman annealing, and an independent
state-constrained HJB benchmark for seasonal storage. The mathematical roles
of relative entropy are kept distinct: the soft Bellman operator provides a
controlled approximation to the hard minimum, while the transition operator
is exactly a KL-penalized worst-case expectation whose Gibbs tilt generates
valid robust SDDP cuts. Water value is then available from two independent
objects, the HJB storage gradient and the HiGHS storage-balance dual.

The numerical evidence supports that architecture while also defining its
limits. The seasonal water-value profiles agree at correlation $0.995$, and an
exact-cut-only replication gives $0.996$; mesh refinement shows that the HJB
water-value derivative is stable even while the absolute value level converges
more slowly. A $3{,}000$-trajectory out-of-sample experiment confirms the
mean--tail trade-off of KL-penalized robustness, the twelve-point sweep verifies
monotonicity of the risk-adjusted value, the paired timing experiment bounds
the value of weekly information, and the two-reservoir cascade reproduces both
marginal water values and their physical ratio. Conversely, the cascade gives
no evidence that entropic initialization accelerates higher-dimensional exact
SDDP. We therefore retain annealing as a smoothing and exploration device and
reserve certification for exact LP-generated cuts.

The next step is empirical rather than conceptual: calibrate the stochastic
inflow and ambiguity model to an operational reservoir system, represent the
actual cascade and operating constraints, and compare against an incumbent
planning workflow. The companion Senegal river hydrology study
\citep{affognon2026wavelet} provides the empirical seasonality layer for that
programme; the present paper provides the optimization and verification layer.

\backmatter

\section*{Acknowledgements}

The authors acknowledge the MOSSHOOS project partners and the institutions
responsible for the hydrological records underpinning the companion empirical
study. The authors also thank colleagues who provided comments on the storage
model, stochastic-optimization formulation, hydrological interpretation,
solver verification, and reproducibility. No site-specific hydrological record
is used in the normalized benchmark analysed in this paper.

\section*{Statements and Declarations}

\textbf{Funding.} This research benefited from the support of the Fondation
Math\'ematique Jacques Hadamard (FMJH) through the Programme Gaspard Monge
pour l'Optimisation, la recherche op\'erationnelle et leurs interactions avec
les sciences des donn\'ees (PGMO).

\textbf{Competing interests.} The authors declare that they have no known
competing financial interests or personal relationships that could have
appeared to influence the work reported in this paper.

\textbf{Ethics approval and consent to participate.} Not applicable.

\textbf{Consent for publication.} Not applicable.

\textbf{Data availability.} The study uses a normalized synthetic
verification benchmark and does not use confidential or person-level
data. All generated numerical data required to reproduce the tables and
figures are included with the accompanying code archive.

\textbf{Code availability.} Python source code, fixed random seeds, and
instructions reproducing the HJB, SDDP, simulation, and figure workflows
are provided in the reproducibility archive accompanying the submission.

\textbf{Author contributions.} Steeven Belvinos Affognon: Conceptualization,
methodology, software, formal analysis, visualization, writing--original
draft. Babacar Mbaye Ndiaye: Supervision, hydropower interpretation,
stochastic-optimization framing, writing--review and editing. 
Cheikh M. F. Kebe: Energy-system interpretation, project supervision,
writing--review and editing. All authors reviewed and approved the final
manuscript. Pierre Mendy: Seasonality methodology, mathematical review, writing--review and editing.

\begin{appendices}

\section{Discretization details}
\label{app:chain}

\subsection*{HJB scheme} Explicit first-order upwind differences for both
advection terms, centered second differences for the degenerate
diffusion $\tfrac12\sigma_0^2 q\,\partial_{qq}V$ (which vanishes at
$q=0$; combined with inward drift $\kappa\theta(t)>0$ no boundary
condition is imposed there beyond one-sided upwinding); at $s=0$ release
is capped by inflow, at $s=S_{\max}$ spill is induced and the storage
drift clipped nonpositive, implementing the constrained viscosity
inequalities. CFL: the sum of Courant numbers is held below $0.9$
($\Delta t = 1/2080$ yr at the baseline $41\times41$ grid; refinement runs tighten $\Delta t$ with the mesh). The periodic fixed point
is a $e^{-\rho}$-contraction of the one-year backward map; a single
Richardson extrapolation after the transient settles reduces ten-plus
cycles to nine at $10^{-5}$ accuracy. (Repeated extrapolation is
unstable because the optimal control feeds back on $V$, making the map
only approximately affine; we document this as a warning.)

\subsection*{Inflow chain} State: weekly-mean inflow, binned into $11$
cells refined near the origin. Transition rows and node values are
estimated from $12{,}000$ SDE paths over two post-burn-in years,
smoothed toward the noncentral-$\chi^2$ CIR rows that are exact for the
frozen-$\theta$ weekly approximation, with $20$ pseudo-counts. The analytic alternative--
bin midpoints as node values--overstates annual water by $34\%$
(midpoints versus conditional means of a right-skewed density) and is
rejected; the empirical chain reproduces the true annual mean to
$0.5\%$. The degrees-of-freedom parameter of the weekly transition,
$d = 4\kappa\theta(t)/\sigma_0^2 = 2F(t)$, drops below $2$ in the
Feller-failing weeks, and the chain inherits the corresponding mass
concentration at the lowest node.

\subsection*{Simulation} Drift-implicit square-root scheme
\citep{alfonsi2005} with per-step guard
$q_k + (\kappa\theta - \tfrac12\sigma_0^2)h \ge 0$; full-truncation
Euler \citep{lord2010} on guard failure. Ten substeps per week. The
simulated state is nonnegative by construction in both branches.

\subsection*{Stage LP and duals} Problem \eqref{eq:stageLP} with $8$
thermal segments whose slopes are segment-midpoint marginal costs (the
piecewise-linear cost coincides with the quadratic at segment edges and
dominates it in between, preserving the lower-bound direction of
phase-A seeds). Duals from HiGHS via \texttt{scipy.optimize.linprog}
(\texttt{method="highs"}); the storage-balance dual was verified against
finite differences of the LP value to $4\times10^{-4}$.

\section{Rigorous foundations, verification, and comparative statics}
\label{app:math}

This appendix collects, for a mathematical readership, the precise
hypotheses under which the objects used in the main text are well defined,
the verification results available for the state-constrained continuous
problem and for its cutting-plane discretization, and proofs of the
comparative-statics facts invoked in Sections~\ref{sec:results}. Numbering
of the results in the main text is retained; new results are numbered in
sequence here.

\subsection{Admissible controls and well-posedness}
\label{app:wellposed}

Fix a filtered probability space $(\Omega,\mathcal F,(\mathcal F_t)_{t\ge0},
\Prob)$ carrying a standard Brownian motion $W$ and satisfying the usual
conditions. Because $q\mapsto\sigma_0\sqrt{q}$ is only $\tfrac12$-H\"older
at the origin, strong uniqueness for \eqref{eq:cir} does not follow from
the Lipschitz theory; it follows instead from the Yamada--Watanabe
criterion, which requires only $\tfrac12$-H\"older continuity of the
diffusion coefficient together with Lipschitz drift
\citep{yamada1971,karatzas1991}. Nonnegativity of $Q$ holds for every
bounded measurable $\theta(\cdot)\ge0$ and is not tied to the Feller
condition; what the Feller condition adds, when it holds uniformly, is
inaccessibility of the origin. Under Assumption~\ref{ass:standing} the
process $Q$ has moments of all orders on compacts, uniformly in $t$.

\begin{definition}[admissible controls]
\label{def:admissible}
Given $(t,s,q)\in[0,1]\times[0,S_{\max}]\times[0,\infty)$, the set
$\mathcal U(s,q)$ consists of pairs $(u,w)$ of progressively measurable
processes with $u_\tau\in[0,U_{\max}]$, $w_\tau\ge0$ for a.e.\ $\tau$, such
that the storage process defined by \eqref{eq:stor} satisfies the state
constraint $S_\tau\in[0,S_{\max}]$ for all $\tau\ge t$ almost surely, and
$\E\int_t^\infty e^{-\rho(\tau-t)}\ell(\tau,u_\tau,w_\tau)\dif\tau<\infty$.
\end{definition}

$\mathcal U(s,q)$ is nonempty for every admissible initial condition: the
policy $u_\tau=\min\{Q_\tau,U_{\max}\}$, $w_\tau=(Q_\tau-u_\tau)
\mathbf 1_{\{S_\tau=S_{\max}\}}$ keeps $S$ constant and has finite
discounted cost because $\ell$ has quadratic growth and $Q$ has second
moments bounded uniformly in time. Since $\ell\ge0$ and $\rho>0$, the value
function \eqref{eq:objective} is finite and satisfies
$0\le V\le \rho^{-1}\sup_t\ell(t,0,0)<\infty$. Joint convexity of $\ell$ in
$(u,w)$ and linearity of \eqref{eq:stor} in $(s,u,w)$ give convexity and
monotonicity of $s\mapsto V(t,s,q)$, as asserted in
Section~\ref{sec:model}; the argument is the standard one, by convex
combination of admissible controls for two initial storages.

\subsection{The state-constrained HJB equation}
\label{app:viscosity}

Because $\mathcal U(s,q)$ depends on the state through the constraint
$S\in[0,S_{\max}]$, the correct notion is Soner's \emph{constrained
viscosity solution}: a viscosity subsolution of \eqref{eq:hjb} on the
closure $[0,S_{\max}]\times[0,\infty)$ and a viscosity supersolution on the
interior \citep{soner1986,capuzzo1990}. The subsolution property on the
closed set encodes the requirement that the controller be able to keep the
state inside the constraint set, and replaces an explicit boundary
condition at $s\in\{0,S_{\max}\}$. At $q=0$ the diffusion coefficient
vanishes while the drift $\kappa\theta(t)$ points strictly inward, so the
origin is an inward-pointing degenerate boundary at which, again, no
boundary condition is imposed; this is the analytic content of the
one-sided upwind treatment described in Appendix~\ref{app:chain}.

Uniqueness among constrained viscosity solutions with polynomial growth
follows from the comparison principle for degenerate second-order
equations on domains satisfying an interior-cone condition, provided the
Hamiltonian is uniformly continuous in the required sense
\citep{crandall1992,soner1986}. Our domain $[0,S_{\max}]\times[0,\infty)$
satisfies the condition, and the discount $\rho>0$ supplies the strict
monotonicity in $V$ needed to close the doubling-of-variables argument.
Convergence of the monotone, stable and consistent finite-difference scheme
of Appendix~\ref{app:chain} to that unique solution is then the
Barles--Souganidis theorem \citep{barles1991}; upwinding is what supplies
monotonicity, and degenerate diffusion is admissible in that framework
\citep{bonnans2003}.

\subsection{Verification}
\label{app:verification}

Two verification results are relevant, and they answer different questions.
The first is the classical sufficient condition; it identifies a candidate
smooth solution with the value function and certifies a candidate control.
The second is an a posteriori certificate computable from the output of the
algorithm, and it is the one actually used in this paper.

\begin{theorem}[classical verification, sufficient form]
\label{thm:verification}
Let $W:[0,1]\times[0,S_{\max}]\times[0,\infty)\to\R$ be $1$-periodic in
$t$, of class $C^{1,1,2}$ on
$(0,1)\times(0,S_{\max})\times(0,\infty)$, continuous up to the boundary,
of polynomial growth in $q$, and suppose:
\begin{enumerate}
\item[(i)] $W$ satisfies \eqref{eq:hjb} pointwise in the interior, with the
constrained inequalities
$\partial_s W\le0$ at $s=S_{\max}$ and the release cap $u\le q$ enforced at
$s=0$;
\item[(ii)] there is a measurable selector
$(u^*,w^*)(t,s,q)$ attaining the minimum in \eqref{eq:hjb} for every
$(t,s,q)$;
\item[(iii)] the closed-loop system obtained by substituting
$(u^*,w^*)$ into \eqref{eq:stor}--\eqref{eq:cir} admits a solution
$(S^*,Q^*)$ with $S^*_\tau\in[0,S_{\max}]$ for all $\tau$, and
$(u^*,w^*)\in\mathcal U(s,q)$.
\end{enumerate}
Then $W=V$ on $[0,1]\times[0,S_{\max}]\times[0,\infty)$ and $(u^*,w^*)$ is
optimal.
\end{theorem}

\begin{proof}
Fix $(t,s,q)$ and an arbitrary $(u,w)\in\mathcal U(s,q)$ with state
$(S,Q)$. Let $\tau_n=\inf\{\tau>t: Q_\tau\ge n \text{ or } Q_\tau\le 1/n\}
\wedge (t+n)$, a localizing sequence. Applying It\^o's formula to
$\tau\mapsto e^{-\rho(\tau-t)}W(\tau,S_\tau,Q_\tau)$ on $[t,\tau_n]$ and
using that $S$ has finite variation (so no second-order term in $s$ and no
cross term appears),
\begin{multline*}
e^{-\rho(\tau_n-t)}W(\tau_n,S_{\tau_n},Q_{\tau_n}) - W(t,s,q) \\
= \int_t^{\tau_n} e^{-\rho(\tau-t)}\Big[
-\rho W + \partial_t W + (Q_\tau-u_\tau-w_\tau)\partial_s W
+ \kappa(\theta(\tau)-Q_\tau)\partial_q W
+ \tfrac12\sigma_0^2 Q_\tau \partial_{qq}W \Big]\dif\tau \\
+ \int_t^{\tau_n} e^{-\rho(\tau-t)}\sigma_0\sqrt{Q_\tau}\,\partial_q W\,\dif W_\tau .
\end{multline*}
By (i) the bracket is $\ge -\ell(\tau,u_\tau,w_\tau)$ for every admissible
$(u,w)$, with equality when $(u,w)=(u^*,w^*)$ by (ii). The stochastic
integral is a true martingale on $[t,\tau_n]$ because $\partial_q W$ is
continuous and $Q$ is bounded on the stopped interval. Taking expectations,
\[
W(t,s,q) \;\le\; \E\!\left[\int_t^{\tau_n} e^{-\rho(\tau-t)}
\ell(\tau,u_\tau,w_\tau)\dif\tau
+ e^{-\rho(\tau_n-t)}W(\tau_n,S_{\tau_n},Q_{\tau_n})\right].
\]
Let $n\to\infty$. Then $\tau_n\to\infty$ a.s.; the polynomial growth of $W$
in $q$, the uniform-in-time moment bounds on $Q$, and $\rho>0$ give
$\E[e^{-\rho(\tau_n-t)}W(\tau_n,S_{\tau_n},Q_{\tau_n})]\to0$ (transversality),
while monotone convergence handles the integral since $\ell\ge0$. Hence
$W(t,s,q)\le J(t,s,q;u,w)$ for every admissible control, so
$W\le V$; and equality holds along $(u^*,w^*)$ by (iii), so $W=V$ and
$(u^*,w^*)$ is optimal.
\end{proof}

\begin{remark}[why Theorem~\ref{thm:verification} is not directly applicable
here]
\label{rem:notsmooth}
Hypothesis (i) requires $C^{1,1,2}$ regularity, which the value function of
the present problem does not possess: $s\mapsto V$ has a kink where the
release constraint becomes active, and $V$ is only semiconcave near
$s\in\{0,S_{\max}\}$. Theorem~\ref{thm:verification} is therefore a
sufficient condition that our $V$ fails to meet, and the identification of
the numerical limit with $V$ proceeds instead through the constrained
viscosity characterization of Appendix~\ref{app:viscosity} plus
Barles--Souganidis convergence. We state
Theorem~\ref{thm:verification} because it is the result a reader will
expect at this point, and because it is exactly what licenses the
threshold policy \eqref{eq:threshold} wherever $V$ \emph{is} smooth -- which
is the interior region where the numerical water values of
Section~\ref{sec:results} are read off. In the viscosity setting the
corresponding statement is the comparison principle: any constrained
viscosity subsolution is dominated by any supersolution, so the numerical
scheme's limit is $V$ without any smoothness assumption.
\end{remark}

The result actually used to certify computed policies is the following
sandwich, which requires no regularity at all and is computable from
quantities the algorithm already produces.

\begin{theorem}[a posteriori verification certificate]
\label{thm:posterior}
Let $\mathcal T$ denote the exact periodic Bellman operator of
\eqref{eq:bellman} and $\mathcal T^{\mathrm{cut}}$ the operator obtained by
replacing each cost-to-go $V_{t+1,j}$ by the upper envelope of a finite
family of cuts, every one of which is valid in the sense of
Lemma~\ref{lem:cut}. Let $\underline V$ be the periodic fixed point of
$\mathcal T^{\mathrm{cut}}$ and let $\pi$ be any admissible policy with
expected discounted cost $\overline V(\pi)$. Then
\[
\underline V \;\le\; V^{\mathrm{chain}} \;\le\; \overline V(\pi),
\]
so $\overline V(\pi)-\underline V$ is a computable optimality certificate:
$\pi$ is $\varepsilon$-optimal for $\varepsilon=\overline V(\pi)-\underline V$.
\end{theorem}

\begin{proof}
Validity of the cuts gives $\mathcal T^{\mathrm{cut}}\le\mathcal T$
pointwise, and both operators are monotone $\delta$-contractions on the
space of bounded functions on the (finite) state set with the sup norm.
Iterating both from the common initial function $0$ and using monotonicity,
$(\mathcal T^{\mathrm{cut}})^k 0\le \mathcal T^k 0$ for every $k$; letting
$k\to\infty$ and using Banach's fixed point theorem in both cases gives
$\underline V\le V^{\mathrm{chain}}$. The right inequality is immediate
because $V^{\mathrm{chain}}$ is an infimum over admissible policies.
\end{proof}

\begin{remark}[what the certificate does and does not cover]
\label{rem:certificate}
Theorem~\ref{thm:posterior} certifies against the \emph{chain} value
$V^{\mathrm{chain}}$, i.e.\ the discrete-time, finite-node model actually
solved. It says nothing about the distance between $V^{\mathrm{chain}}$ and
the continuous-time $V$; that distance is what the HJB benchmark of
Section~\ref{sec:results} measures, and
Section~\ref{sec:timing} decomposes it into an information component and a
discretization component. The two instruments are complementary:
Theorem~\ref{thm:posterior} is an internal certificate,
the HJB comparison is an external one. This is the precise sense in which
the architecture is ``certified rather than merely tested''.
\end{remark}

\begin{remark}[the entropic phase is not covered by
Theorem~\ref{thm:posterior}]
\label{rem:seedvalidity}
Proposition~\ref{prop:soft}(i) states $\mathcal T\le\mathcal T^\eta$: the
entropic operator \emph{over}estimates. Affine minorants extracted from a
phase-A value function are therefore minorants of $V^\eta$, not
necessarily of $V^{\mathrm{chain}}$, and the discrepancy is bounded by
$\eta\log N_u/(1-\delta)$ plus the phase-A discretization error, of which
only the first term is controlled a priori. Consequently the entropic
phase must be used for \emph{guidance} -- choosing the trial states at
which exact cuts are generated -- and not as a source of certified cuts,
if the bound of Theorem~\ref{thm:posterior} is to remain valid. This is the
design adopted in Section~\ref{sec:cascade}, and the single-reservoir
seeded bounds reported in Section~\ref{sec:results} are accompanied by the
a posteriori check described there.
\end{remark}

\subsection{Comparative statics in the robustness coefficient}
\label{app:gamma}

Write $\rho_\gamma(v)=\gamma^{-1}\log\E_P[e^{\gamma v}]$ for a random
variable $v$ taking finitely many values, and $\rho_0(v)=\E_P[v]$.

\begin{proposition}[monotonicity and expansion of the entropic risk]
\label{prop:rhomono}
For each fixed $v$, the map $\gamma\mapsto\rho_\gamma(v)$ is nondecreasing
and continuous on $[0,\infty)$, with
\[
\rho_\gamma(v)\;=\;\E_P[v]+\frac{\gamma}{2}\operatorname{Var}_P(v)+O(\gamma^2),
\qquad \gamma\downarrow0 .
\]
\end{proposition}

\begin{proof}
Let $K(\gamma)=\log\E_P[e^{\gamma v}]$ be the cumulant generating function,
so $\rho_\gamma=K(\gamma)/\gamma$ and $K(0)=0$. $K$ is convex and smooth on
$\R$ (finitely many values), hence $K(0)\ge K(\gamma)+K'(\gamma)(0-\gamma)$,
i.e.\ $\gamma K'(\gamma)-K(\gamma)\ge0$. Therefore
$\rho_\gamma' = \big(\gamma K'(\gamma)-K(\gamma)\big)/\gamma^2 \ge 0$.
Continuity at $0$ and the expansion follow from
$K(\gamma)=\gamma\E[v]+\tfrac{\gamma^2}{2}\operatorname{Var}(v)+O(\gamma^3)$.
\end{proof}

\begin{proposition}[monotonicity of the value in $\gamma$]
\label{prop:Vmono}
Let $V^{(\gamma)}$ be the periodic fixed point of the Bellman operator
\eqref{eq:bellman} with $\mathcal R=\rho_\gamma$. Then
$\gamma\mapsto V^{(\gamma)}$ is nondecreasing pointwise.
\end{proposition}

\begin{proof}
By Proposition~\ref{prop:rhomono}, $\gamma_1\le\gamma_2$ implies
$\rho_{\gamma_1}\le\rho_{\gamma_2}$ pointwise, hence the corresponding
Bellman operators satisfy $\mathcal T^{(\gamma_1)}\le\mathcal T^{(\gamma_2)}$
pointwise; both are monotone $\delta$-contractions. Iterating from the
common initial function $0$ preserves the inequality at every step, and
passing to the limit gives $V^{(\gamma_1)}\le V^{(\gamma_2)}$.
\end{proof}

\begin{remark}[monotonicity of the value does not transfer to its gradient]
\label{rem:ssvmono}
Proposition~\ref{prop:Vmono} concerns the value, not the storage
derivative $\ssv=-\partial_s V^{(\gamma)}$. No monotonicity of $\ssv$ in
$\gamma$ is claimed or should be expected: a family of convex functions
increasing pointwise need not have monotone slopes. The numerical
comparative statics of Section~\ref{sec:gammasweep} are consistent with
this: the value is strictly increasing and the season-averaged storage
value is strictly increasing, while individual weeks exhibit small
reversals, roughly half of which lie inside the seed-to-seed variability of
the linear-programming dual.
\end{remark}

\subsection{The information ordering of the timing conventions}
\label{app:timing}

Let $V^{\mathrm{DH}}_t(s,j)$ and $V^{\mathrm{HD}}_t(s,i)$ denote the
periodic fixed points of the decision--hazard and hazard--decision
recursions defined in Section~\ref{sec:timing}, on the same chain, with the
same costs and the same storage set.

\begin{proposition}[the information advantage is nonnegative]
\label{prop:info}
For every $t$, every $s\in[0,S_{\max}]$ and every node $i$,
\[
\Delta_{\mathrm{info}}(t,s,i)\;:=\;V^{\mathrm{HD}}_t(s,i)
-\sum_{j}P_{t-1}[i,j]\,V^{\mathrm{DH}}_t(s,j)\;\ge\;0 .
\]
\end{proposition}

\begin{proof}
Let $\Phi^{\mathrm{DH}},\Phi^{\mathrm{HD}}$ be the two one-sweep operators.
Suppose $V'^{\mathrm{HD}}_{t+1}(\cdot,j)\ge\sum_{j'}P_t[j,j']
V'^{\mathrm{DH}}_{t+1}(\cdot,j')$ for all $j$. Writing
$G_j(u)=c_t(u,q_{t,j})+\delta\sum_{j'}P_t[j,j']
V'^{\mathrm{DH}}_{t+1}(s'_j(u),j')$,
\begin{align*}
(\Phi^{\mathrm{HD}}V')_t(s,i)
&=\min_{u}\sum_j P_{t-1}[i,j]\Big\{c_t(u,q_{t,j})
+\delta V'^{\mathrm{HD}}_{t+1}(s'_j(u),j)\Big\}\\
&\ge\min_{u}\sum_j P_{t-1}[i,j]\,G_j(u)\\
&\ge\sum_j P_{t-1}[i,j]\min_{u}G_j(u)\\
&=\sum_j P_{t-1}[i,j]\,(\Phi^{\mathrm{DH}}V')_t(s,j),
\end{align*}
the middle inequality being the interchange
$\min_u\E\ge\E\min_u$. The hypothesis holds trivially for the initial
function $0$; since both operators are monotone $\delta$-contractions, the
inequality is preserved along the iteration and passes to the periodic
fixed points.
\end{proof}

Proposition~\ref{prop:info} is what makes the paired experiment of
Section~\ref{sec:timing} a \emph{decomposition} rather than a comparison of
two unrelated numbers: $\Delta_{\mathrm{info}}$ is a nonnegative quantity
computed on the identical chain, so it contains no discretization or
aggregation component, and the residual
$(V^{\mathrm{HJB}}-V^{\mathrm{DH}})-\Delta_{\mathrm{info}}$ is
attributable to time and space discretization. The numerical values in
Section~\ref{sec:timing} satisfy the predicted sign everywhere
(minimum $4.5\times10^{-2}$ over the whole grid).

\subsection{Proof of Proposition~\ref{prop:soft}, with the annealing bound}
\label{app:softproof}

For completeness we restate the argument of Section~\ref{sec:soft} with the
constants made explicit. Let $\pi_0$ be uniform on a control grid with
$N_u$ points and let $G$ be bounded below.

Since $\min_u G\le G(u)$ for all $u$,
$\E_{\pi_0}e^{-G/\eta}\le e^{-\min_u G/\eta}$, and taking $-\eta\log$
reverses the inequality, giving $\mathcal T^\eta\ge\mathcal T$. For the
upper bound, retaining only the minimizing grid point,
$\E_{\pi_0}e^{-G/\eta}\ge N_u^{-1}e^{-\min_u G/\eta}$, whence
$\mathcal T^\eta\le\mathcal T+\eta\log N_u$. Monotonicity is inherited from
that of $G\mapsto-\eta\log\E_{\pi_0}e^{-G/\eta}$, and the contraction
modulus is $\delta$ because adding a constant $c$ to $V$ adds $\delta c$ to
$G$ and $-\eta\log\E e^{-(G+\delta c)/\eta}=\delta c
-\eta\log\E e^{-G/\eta}$. Finally, if $V^*=\mathcal T V^*$ and
$V^\eta=\mathcal T^\eta V^\eta$ then
\begin{align*}
\|V^\eta-V^*\|_\infty
&=\|\mathcal T^\eta V^\eta-\mathcal T V^*\|_\infty\\
&\le \|\mathcal T^\eta V^\eta-\mathcal T^\eta V^*\|_\infty
 +\|\mathcal T^\eta V^*-\mathcal T V^*\|_\infty\\
&\le \delta\|V^\eta-V^*\|_\infty+\eta\log N_u.
\end{align*}
so $0\le V^\eta-V^*\le\eta\log N_u/(1-\delta)$, the lower bound coming from
$\mathcal T^\eta\ge\mathcal T$ and monotonicity.

It is worth being explicit about the size of this constant in the periodic
setting, because it is the reason for a design choice in
Section~\ref{sec:anneal}. With $N_u=61$ and $\rho=0.1$ on weekly stages,
$1-\delta=1-e^{-\rho/52}=1.92\times10^{-3}$, so the bound equals
$21.4$, $6.4$ and $2.1$ at the three annealing temperatures
$\eta\in\{10^{-2},3\times10^{-3},10^{-3}\}$ -- all far larger than the
values themselves, which are of order unity. The weak discounting that
makes seasonal storage interesting simultaneously destroys the
informativeness of the uniform annealing bound. Proposition~\ref{prop:soft}
therefore justifies annealing \emph{qualitatively} -- the entropic fixed
point converges monotonically down to the exact one as $\eta\downarrow0$,
and the approximation is one-sided -- but does not supply a useful
numerical certificate at these discount rates. This is why phase A is
terminated with hard sweeps rather than stopped at small $\eta$, why the
entropic phase is used for guidance only
(Remark~\ref{rem:seedvalidity}), and why certification is delegated to
Theorem~\ref{thm:posterior} and to the external HJB comparison.

\subsection{Proof of Lemma~\ref{lem:cut} in the risk-neutral limit}
\label{app:cutlimit}

Taking $\gamma\downarrow0$ in \eqref{eq:tilt} gives
$\tilde w\to P_t[j,\cdot]$ by continuity, and
$\rho_\gamma(v)\to\E_P[v]$ by Proposition~\ref{prop:rhomono}. The
statement of Lemma~\ref{lem:cut} therefore specializes to the classical
Benders cut of risk-neutral SDDP, $\bar b=\sum_{j'}P_t[j,j']\beta_{j'}$ and
$\bar a=\sum_{j'}P_t[j,j']v_{j'}-\bar b\hat s$, whose validity is the usual
consequence of convexity. The content of the lemma is thus that the
entropic tilt is the correct probability reweighting to apply to the dual
vector, and that no additional correction term is needed--a consequence
of the fact that $\tilde w$ is the Donsker--Varadhan maximizer for the
value vector $v$, which is exactly the vector at which the subgradients
$\beta$ are evaluated.

\end{appendices}

\bibliography{inflow_refs}

@article{pereira1991,
  author  = {Pereira, M. V. F. and Pinto, L. M. V. G.},
  title   = {Multi-stage stochastic optimization applied to energy planning},
  journal = {Mathematical Programming},
  volume  = {52},
  pages   = {359--375},
  year    = {1991}
}

@article{crandall1992,
  author  = {Crandall, Michael G. and Ishii, Hitoshi and Lions, Pierre-Louis},
  title   = {User's guide to viscosity solutions of second order partial differential equations},
  journal = {Bulletin of the American Mathematical Society},
  volume  = {27},
  number  = {1},
  pages   = {1--67},
  year    = {1992}
}

@book{fleming2006,
  author    = {Fleming, Wendell H. and Soner, H. Mete},
  title     = {Controlled Markov Processes and Viscosity Solutions},
  edition   = {2nd},
  publisher = {Springer},
  address   = {New York},
  year      = {2006}
}

@article{cir1985,
  author  = {Cox, John C. and Ingersoll, Jonathan E. and Ross, Stephen A.},
  title   = {A theory of the term structure of interest rates},
  journal = {Econometrica},
  volume  = {53},
  number  = {2},
  pages   = {385--407},
  year    = {1985}
}

@article{feller1951,
  author  = {Feller, William},
  title   = {Two singular diffusion problems},
  journal = {Annals of Mathematics},
  volume  = {54},
  number  = {1},
  pages   = {173--182},
  year    = {1951}
}

@article{alfonsi2005,
  author  = {Alfonsi, Aur\'elien},
  title   = {On the discretization schemes for the {CIR} (and {B}essel squared) processes},
  journal = {Monte Carlo Methods and Applications},
  volume  = {11},
  number  = {4},
  pages   = {355--384},
  year    = {2005}
}

@article{lord2010,
  author  = {Lord, Roger and Koekkoek, Remmert and van Dijk, Dick},
  title   = {A comparison of biased simulation schemes for stochastic volatility models},
  journal = {Quantitative Finance},
  volume  = {10},
  number  = {2},
  pages   = {177--194},
  year    = {2010}
}

@book{dupuis1997,
  author    = {Dupuis, Paul and Ellis, Richard S.},
  title     = {A Weak Convergence Approach to the Theory of Large Deviations},
  publisher = {Wiley},
  address   = {New York},
  year      = {1997}
}

@article{todorov2009,
  author  = {Todorov, Emanuel},
  title   = {Efficient computation of optimal actions},
  journal = {Proceedings of the National Academy of Sciences},
  volume  = {106},
  number  = {28},
  pages   = {11478--11483},
  year    = {2009}
}

@article{kappen2005,
  author  = {Kappen, Hilbert J.},
  title   = {Path integrals and symmetry breaking for optimal control theory},
  journal = {Journal of Statistical Mechanics: Theory and Experiment},
  volume  = {2005},
  number  = {11},
  pages   = {P11011},
  year    = {2005}
}

@inproceedings{geist2019,
  author    = {Geist, Matthieu and Scherrer, Bruno and Pietquin, Olivier},
  title     = {A theory of regularized {M}arkov decision processes},
  booktitle = {Proceedings of the 36th International Conference on Machine Learning},
  series    = {PMLR},
  volume    = {97},
  pages     = {2160--2169},
  year      = {2019}
}

@article{kupper2009,
  author  = {Kupper, Michael and Schachermayer, Walter},
  title   = {Representation results for law invariant time consistent functions},
  journal = {Mathematics and Financial Economics},
  volume  = {2},
  number  = {3},
  pages   = {189--210},
  year    = {2009}
}

@article{ruszczynski2010,
  author  = {Ruszczy\'nski, Andrzej},
  title   = {Risk-averse dynamic programming for {M}arkov decision processes},
  journal = {Mathematical Programming},
  volume  = {125},
  pages   = {235--261},
  year    = {2010}
}

@article{shapiro2011,
  author  = {Shapiro, Alexander},
  title   = {Analysis of stochastic dual dynamic programming method},
  journal = {European Journal of Operational Research},
  volume  = {209},
  number  = {1},
  pages   = {63--72},
  year    = {2011}
}

@article{philpott2013,
  author  = {Philpott, Andy B. and de Matos, Vitor L.},
  title   = {Dynamic sampling algorithms for multi-stage stochastic programs with risk aversion},
  journal = {European Journal of Operational Research},
  volume  = {218},
  number  = {2},
  pages   = {470--483},
  year    = {2012}
}

@book{hansen2008,
  author    = {Hansen, Lars Peter and Sargent, Thomas J.},
  title     = {Robustness},
  publisher = {Princeton University Press},
  address   = {Princeton},
  year      = {2008}
}

@article{shapiro2020,
  author  = {Shapiro, Alexander and Ding, Lingquan},
  title   = {Periodical multistage stochastic programs},
  journal = {SIAM Journal on Optimization},
  volume  = {30},
  number  = {3},
  pages   = {2083--2102},
  year    = {2020}
}

@article{huangfu2018,
  author  = {Huangfu, Qi and Hall, J. A. Julian},
  title   = {Parallelizing the dual revised simplex method},
  journal = {Mathematical Programming Computation},
  volume  = {10},
  number  = {1},
  pages   = {119--142},
  year    = {2018}
}

@article{philpott2008,
  author  = {Philpott, Andy B. and Guan, Ziming},
  title   = {On the convergence of stochastic dual dynamic programming and related methods},
  journal = {Operations Research Letters},
  volume  = {36},
  number  = {4},
  pages   = {450--455},
  year    = {2008}
}

@article{girardeau2015,
  author  = {Girardeau, Pierre and Lecl\`ere, Vincent and Philpott, Andy B.},
  title   = {On the convergence of decomposition methods for multistage stochastic convex programs},
  journal = {Mathematics of Operations Research},
  volume  = {40},
  number  = {1},
  pages   = {130--145},
  year    = {2015}
}

@book{follmer2016,
  author    = {F\"ollmer, Hans and Schied, Alexander},
  title     = {Stochastic Finance: An Introduction in Discrete Time},
  edition   = {4th},
  publisher = {De Gruyter},
  address   = {Berlin},
  year      = {2016}
}

@article{soner1986,
  author  = {Soner, H. Mete},
  title   = {Optimal control with state-space constraint. {I}},
  journal = {SIAM Journal on Control and Optimization},
  volume  = {24},
  number  = {3},
  pages   = {552--561},
  year    = {1986}
}

@article{capuzzo1990,
  author  = {Capuzzo-Dolcetta, Italo and Lions, Pierre-Louis},
  title   = {Hamilton--{J}acobi equations with state constraints},
  journal = {Transactions of the American Mathematical Society},
  volume  = {318},
  number  = {2},
  pages   = {643--683},
  year    = {1990}
}

@article{barles1991,
  author  = {Barles, Guy and Souganidis, Panagiotis E.},
  title   = {Convergence of approximation schemes for fully nonlinear second order equations},
  journal = {Asymptotic Analysis},
  volume  = {4},
  number  = {3},
  pages   = {271--283},
  year    = {1991}
}

@article{bonnans2003,
  author  = {Bonnans, J. Fr\'ed\'eric and Zidani, Hasnaa},
  title   = {Consistency of generalized finite difference schemes for the stochastic {HJB} equation},
  journal = {SIAM Journal on Numerical Analysis},
  volume  = {41},
  number  = {3},
  pages   = {1008--1021},
  year    = {2003}
}

@article{yamada1971,
  author  = {Yamada, Toshio and Watanabe, Shinzo},
  title   = {On the uniqueness of solutions of stochastic differential equations},
  journal = {Journal of Mathematics of Kyoto University},
  volume  = {11},
  number  = {1},
  pages   = {155--167},
  year    = {1971}
}

@book{karatzas1991,
  author    = {Karatzas, Ioannis and Shreve, Steven E.},
  title     = {Brownian Motion and Stochastic Calculus},
  edition   = {2nd},
  publisher = {Springer},
  address   = {New York},
  year      = {1991}
}

@article{lam2019,
  author  = {Lam, Henry},
  title   = {Recovering best statistical guarantees via the empirical divergence-based distributionally robust optimization},
  journal = {Operations Research},
  volume  = {67},
  number  = {4},
  pages   = {1090--1105},
  year    = {2019}
}

@article{duchi2021,
  author  = {Duchi, John C. and Namkoong, Hongseok},
  title   = {Learning models with uniform performance via distributionally robust optimization},
  journal = {The Annals of Statistics},
  volume  = {49},
  number  = {3},
  pages   = {1378--1406},
  year    = {2021}
}

@article{schwartz1997,
  author  = {Schwartz, Eduardo S.},
  title   = {The stochastic behavior of commodity prices: implications for valuation and hedging},
  journal = {The Journal of Finance},
  volume  = {52},
  number  = {3},
  pages   = {923--973},
  year    = {1997}
}

@article{gjelsvik1992,
  author  = {Gjelsvik, Anders and Ro\o{}tting, Trond A. and Roynstrand, Jarand},
  title   = {Long-term scheduling of hydro-thermal power systems},
  journal = {Hydropower '92},
  pages   = {539--546},
  year    = {1992}
}

@book{loucks2017,
  author    = {Loucks, Daniel P. and van Beek, Eelco},
  title     = {Water Resource Systems Planning and Management: An Introduction to Methods, Models, and Applications},
  publisher = {Springer},
  address   = {Cham},
  year      = {2017}
}

@misc{affognon2026wavelet,
  author        = {Affognon, Steeven B. and Ndiaye, Babacar M. and Mendy, Pierre and Kebe, Cheikh M. F.},
  title         = {From Daily Fluctuations to Annual Hydrological Cycles: A Wavelet-Based Analysis of Nonstationary Seasonality in Senegal River Hydropower Inflows},
  year          = {2026},
  eprint        = {2608.23470},
  archivePrefix = {arXiv},
  primaryClass  = {stat.AP},
  note          = {arXiv:2608.23470}
}

@article{dowson2025convex,
  author  = {Dowson, Oscar and Morton, David P. and Pagnoncelli, Bernardo K.},
  title   = {Incorporating convex risk measures into multistage stochastic programming algorithms},
  journal = {Annals of Operations Research},
  volume  = {348},
  number  = {2},
  pages   = {807--831},
  year    = {2025},
  doi     = {10.1007/s10479-022-04977-w}
}

@article{pavicevic2026water,
  author  = {Pavi\v{c}evi\'{c}, Matija and Huallpara, Alizon and Yu, A. and Balderrama, Sergio and Ploussard, Quentin},
  title   = {Water value and market response functions in hydropower-dominated electricity systems: A systematic review and methodological comparison},
  journal = {Renewable and Sustainable Energy Reviews},
  volume  = {235},
  pages   = {116952},
  year    = {2026},
  doi     = {10.1016/j.rser.2026.116952}
}

@article{shapiro2013risk,
  author  = {Shapiro, Alexander and Tekaya, Wajdi and da Costa, Joari Paulo and Soares, Murilo Pereira},
  title   = {Risk neutral and risk averse stochastic dual dynamic programming method},
  journal = {European Journal of Operational Research},
  volume  = {224},
  number  = {2},
  pages   = {375--391},
  year    = {2013},
  doi     = {10.1016/j.ejor.2012.08.022}
}

\end{document}